%% file: main.tex
\documentclass[12pt]{article}
\usepackage{graphicx} 
\usepackage{url}
\usepackage{amsmath}
\usepackage{amsthm}
\usepackage{amsfonts}
\usepackage{listings}
\usepackage{geometry}
\usepackage{tikz}
\usepackage{todonotes}
\usepackage{marginnote}
\usepackage{etoolbox}
\newtoggle{lmargin}
\newcommand{\alternatingtodo}[2][]{%
    \iftoggle{lmargin}%
    {%
        \todo[#1]{#2}%
        \togglefalse{lmargin}%
    }{%
        {%
            \let\marginpar\marginnote%
            \reversemarginpar%
            \todo[#1]{#2}%
        }%
        \toggletrue{lmargin}%
    }%
    \ignorespaces%
}
\usepackage[normalem]{ulem}
\usepackage{tikz-cd} 
\usepackage{enumitem}
\usepackage{subfigure}  
\usepackage{bbm}
\usepackage{natbib}
\usepackage{xcolor}
\usepackage{setspace}   
\usepackage{pdfpages}
\usepackage{subcaption}
\usepackage[export]{adjustbox}
\usepackage{wrapfig}

\newtheorem{lemma}{Lemma}
\newtheorem{theorem}{Theorem}

\newtheorem{Proposition}{Proposition}
\theoremstyle{definition}

\newtheorem{Example}{Example}
\newtheorem{Remark}{Remark}

\def\E{\mathbb{E}}
\def\EP#1#2{\mathbb{E}_{#1}{\left(#2\right)}}
\def\P#1{{\mathbb{P}}\left(#1\right)}

\def\calP{{\mathcal{P}}}

\def\calL{{\mathcal{L}}}
\def\calB{{\mathcal{B}}}
\def\calC{{\mathcal{C}}}
\def\calF{{\mathcal{F}}}

\def \R{\mathbf{R}}

\def\Real{\mathbb{R}}

\title{On a Class of Optimal Reinsurance Problems}
\author{N. D. Shyamalkumar and Tianrun Wang\thanks{Corresponding author: tianrun-wang@uiowa.edu}}

\input{math.tex}

\begin{document}
\maketitle

\begin{abstract}
De Finetti's optimal reinsurance is a set of contracts, one for each risk in a portfolio, that caps the retained aggregate variance to a pre-specified level while minimizing total expected loss. The premiums are determined using the expected value principle, and the safety loading is allowed to vary with the risks. The original formulation assumed that the risks were independent and restricted contracts to quota shares on individual risks. A recent variation surprisingly yields a closed form for the contracts, while allowing dependence between risks and permitting the contracts to depend on all risks, without restricting their functional form. We extend this to the case of an arbitrary convex functional as the risk measure and use duality tools from convex analysis to show the equivalence between the constrained and the penalized versions of the underlying optimization problem. To explicitly solve the penalized version for the variance and the conditional value at risk (CVaR) as the risk measure, we resort to either variational analysis or a rudimentary approach. We show that a rudimentary approach can also address the choice of VaR, a non-convex functional, as the risk measure. 
\end{abstract}

\section{Introduction}

A common way for an insurer to alter the risk profile of its insurance liabilities is through risk transfer via a reinsurance treaty. The desire for a change in the risk profile may be driven by capital considerations, or to gain the ability to write larger policies ({\it i.e.}, capacity expansion needs). Also, for expansion into new and complex products, the purchase of reinsurance may provide access to the reinsurer's significant expertise in pricing, and with reduced risk exposure, the insurer can gain confidence to expand into new products and markets. 

Once an insurer decides to reinsure and defines a preference ordering over reinsurance treaties within a feasible set, it is natural to seek an optimal reinsurance treaty.
In \citet{de1940problema}, the classical de Finetti optimal reinsurance problem, quota-share contracts form the feasible set of reinsurance treaties, and two such treaties are ordered in terms of the variance of the retained aggregate loss. In \citet{borch1960safety,borch1960attempt} and \citet{arrow1963uncertainty}, different preference orderings and feasible sets of insurance treaties were considered, and remarkably, without restricting the reinsurance indemnity function to a parametric class, they established the optimality of simple forms of indemnity functions. Since then, academic interest in optimal reinsurance has sustained and, in recent years, grown significantly. See \citet{albrecher2017reinsurance} for an excellent recent overview of the literature.   

A mathematically elegant relaxation allows for randomized reinsurance contracts. Such contracts allow for the retained amount to be determined through a random experiment conducted with the knowledge of the claim sizes. While there exist significant hurdles for the marketplace adoption of such reinsurance contracts, the sub-optimality of deterministic contracts in some settings and the fact that reinsurance always entails the assumption of counterparty default risk are two supporting arguments for its consideration - see, \citet{albrecher2019randomized, ASIMIT2013690, VINCENT2021125, ASIMIT2013690}, and \citet{CAI201413}. Mathematically, this relaxation allows one to pose the optimal reinsurance problem as the minimization of a lower semi-continuous law-invariant risk functional, say $\calP$, over a compact set of probability measures, where these probability measures represent the joint distribution of the risks and reinsured amounts. 
 
This paper is motivated by \citet{acciaio2025optimal}, which considers a class of optimal reinsurance problems in the multiple-risk setting that allow for randomized reinsurance contracts. In the case where there are finitely many constraints defined by a sub-level set of a lower semi-continuous vector-valued function, say $\calG$, \citet{acciaio2025optimal} are able to characterize the support of the optimal distribution using the assumed convex-linearity of the directional derivatives of $\calP$ and $\calG$ on a suitably chosen convex set, say $\calC$, containing the optimal distribution. Interestingly, as illustrated through examples, even the choice of a relatively small $\calC$ can provide useful characterizations of the support of the optimal distribution. One such example that particularly piqued our interest is their Example 6.1, which generalizes the classical de Finetti reinsurance problem (see \citet{de1940problema}).

In the formulation of the classical de Finetti reinsurance problem, the risk measure is variance, the risk $X_i$s are assumed to be independent, and $R_i$s are restricted to quota-share contracts, {\it i.e. $R_i=a_i X_i$} for $a_i\in [0,1]$.  A much less constrained formulation is studied in Example 6.1, wherein the $n$ risks may be dependent, and $R_i$s are no longer constrained to any specific functional form, and moreover need not even be measurable with respect to $\sigma\langle X_i \rangle$.  In particular, $R_i$s could depend on all of the $n$ risks and also additional sources of information and randomness, and hence need not be measurable with respect to $\sigma\langle X_1, \cdots, X_n \rangle$. This flexibility can not only account for the above-mentioned reinsurer counterparty default risk but also accommodate contracts with coverage linked to investment performance, as in some finite-risk reinsurance contracts (see \citet{culp2011structured} and \citet{albrecher2019randomized}). 

A natural question is whether the formulation of the above generalized version of the de Finetti reinsurance problem in terms of random vectors representing claim sizes and reinsured amounts would lead to a simpler analysis. To this end, we consider an insurer with a portfolio consisting of $n$ insurance risks $\X=(X_1, \ldots, X_n)$, seeking a reinsurance contract $\R = (R_1, \ldots, R_n)$ that minimizes the aggregate expected loss after reinsurance, under a constraint that a risk measure of the retained aggregate loss does not exceed a given threshold. Specifically, let $\rho$ be a risk measure, the insurer under the constraint that 
\[
\rho\left(\sum_{i=1}^{n} (X_i - R_i) \right) \leq c,
\]
seeks a reinsurance contract that minimizes the expected value of the loss, given by 
\[
\sum_{i=1}^{n} \left(X_i - R_i + (1 + \beta_i) \E{R_i}\right),
\]
where $\beta_i$ is the safety loading for the $i$-th contract, {\it i.e.} a surcharge by the reinsurer over pure premium for assuming the risk. We denote the underlying probability space by  $(\Omega,\calF,\mathbb{P})$. We let the risk measure $\rho:\calL^1((\Omega,\calF,\mathbb{P})) \to \Real \cup \pm \infty$. We note that the above optimization problem is feasible for $\mathbf{X} \in \calL^1(\Omega,\calF,\mathbb{P})$ provided $\rho(0)\leq c$, even in the case that $\rho$ equals the variance functional. Note that the variance is a concave functional on the space of measures, but convex on the space of random vectors. Note that a convex risk measure, as introduced by \citet{follmer2002convex}, requires, in addition to convexity, both monotonicity and translation equivariance. While standard deviation and variance are convex, they are translation invariant and not equivariant. Variance, moreover, is not positively homogeneous.

Our contributions are several-fold. We allow the risk functional $\rho$ to be convex, and lower semi-continuous on $\calL^1((\Omega,\calF,\mathbb{P}))$. When $\rho$ is the variance, we work with the weak topology, and when $\rho$ equals the conditional value-at-risk (CVaR), we employ the Wasserstein-1 distance. In both cases, the problem and its dual have closed-form solutions; we derive closed-form solutions for the optimal reinsurance. 
We note that we don't require the risk $\X$ to have a distribution that is absolutely continuous with respect to the Lebesgue measure. Finally, we derive a closed-form solution for the value-at-risk (VaR) risk functional under the assumption that the probability space is atomless - a mild assumption satisfied if we have an independent exogenous random number generator with a non-atomic distribution.

 {\textbf{Notations:}} We assume $\frac{x}{0}=+\infty$ for any $x\geq 0$ and $0 \cdot \infty=0$. Binary relations between two events, or two random vectors will mean in the $\mathbb{P}$ {\em almost sure} sense. For example, for $A,\, B \in \calF$, by $A\subseteq B$ we mean $\mathbb{P}(A\cap B^{\sf c})=0$, and for random variables $W_1,\, W_2$, we say $W_1=W_2$ to indicate $W_1=W_2$ except on a $\mathbb{P}-$null set. For any $k\geq 1$, we use $\xrightarrow[]{d}$ to denote weak convergence, and use bounded and continuous test functions $\phi: \Real^{k} \to \Real$ to define the weak topology on $\calL^1(\Omega,\calF,\mathbb{P})^k$, {\it i.e.} $\mathbf{W_n} \xrightarrow[]{d} \mathbf{W}$ if for all bounded and continuous test functions $\phi: \Real^{k} \to \Real$,  $\E{\phi(\mathbf{W_n})} \to  \E{\phi(\mathbf{W})}$.

\section{{Optimal Reinsurance Problem - Two Versions}}
We define the following two functions, both from $\calL^1(\Omega,\calF,\mathbb{P})^{2n}$ to $\Real$:
\begin{equation}
    \label{f-def}
    f(\R,\X):=\sum_{i=1}^n \beta_i \E{R_i},
\end{equation}
and
\begin{equation}
\label{g-def}
    g(\R,\X):= \rho\left(\sum_{i=1}^{n} (X_i - R_i) \right).
\end{equation}
We begin by observing that the problem as posed is the following constrained optimization problem:
\begin{equation}
    \label{Reins-C}
\begin{aligned}
C(c) := \arg\min_{\R} \quad & { f(\R,\X)} \\
\text{s.t.} \quad & 0 \leq R_i \leq X_i,\ \forall i \\
\text{and}\quad                  & { g(\R,\X)} \leq c
\end{aligned}
    \tag{Reins-C}
\end{equation}

We can see that when $c  \geq \rho \left( \sum_{i=1}^n X_i  \right)$, $C(c)=\{\mathbf{0}\}$ trivially.
We define the penalized version \ref{Reins-P} to be the following:
\begin{equation}
    \label{Reins-P}
    \begin{aligned}
P(\lambda) := \arg\min_{\R} \quad & { f(\R,\X)}+\lambda { g(\R,\X)} \\
\text{s.t.} \quad & 0 \leq R_i \leq X_i,\ \forall i 
\end{aligned}
    \tag{Reins-P}
\end{equation}
We note that the sets $C(c)$ and $P(\lambda)$ may be non-singleton, in which case there is no exact correspondence between $C(c)$ and $P(\lambda)$. However, we have the following theorem connecting $C(c)$ and $P(\lambda)$.

\begin{theorem}
\label{dual}
For $\X\in L^1((\Omega,\calF,\mathbb{P}))^{\times n}$,  let $\calE\subseteq L^1(\Omega,\calF,\mathbb{P})$ be defined by 
\[
\calE:=\left\{\sum_{i=1}^n (X_i - R_i):0 \leq R_i \leq X_i, \forall i \right\}.
\]
Also, let the risk functional $\rho:\calE \to \Real \cup \{\pm \infty\}$ be convex, and weakly lower semi-continuous. Then, we have the following:
\begin{enumerate}[label=\roman*.]
    \item For any $c \geq \rho(0)$ and $\lambda \geq 0$, both $C(c)$ and $P(\lambda)$ are non-empty.
    \item Moreover, for any $\rho(0) < c <\rho \left( \sum_{i=1}^n X_i  \right)$, there exists $\lambda^* > 0$ such that $C(c) \subset P(\lambda^*)$.
    \item For any $\lambda>0$,
\[
P(\lambda) \cap \left\{ \R: { g(\R,\X)}=c \right\}=C(c) \text{ or } \emptyset
\]
\end{enumerate}  
\end{theorem}

\begin{proof} Consider the convex set $\calR$ defined by, 
\[
\calR:=\{ (\R,\X):0 \leq R_i \leq X_i, \, i =1,\ldots,n\} \subset L^1(\Omega,\calF,\mathbb{P})^{\times 2n}.
\]
We choose to endow $\calR$, to be precise, the space of equivalence sets defined by distributional equivalence, with the weak topology. In other words, $(\R_k,\X)$ converges to $(\R,\X)$ on $\calR$ if these random vectors converge weakly.  
$\calR$ is sequentially precompact, by tightness and Prokhorov's theorem (see Theorem 5.1 of \citet{billingsley1999convergence}). To show $\calR$ is closed, we observe that its corresponding measures are those with an $n$-marginal equal to $\calL(\X)$ and, by rearrangement of coordinates, a support which is an $n$-product of the closed subset of a first quadrant of the Euclidean plane containing the points below the diagonal. With this characterization, and by using the Portmanteau theorem (see \citet{billingsley1999convergence}), we see that $\calR$ is sequentially closed. But since the weak topology on an Euclidean space is metrizable, we have $\calR$ is compact.

\noindent {\bf Proof of i.:}  Note that, in view of the above properties of $\calR$, to show that both $C(c)$ and $P(\lambda)$ are nonempty it suffices to show that ${ f(\R,\X)}$ is linear and continuous with respect to the weak topology, and ${ g(\R,\X)}$ is convex, and lower semi-continuous with respect to the weak topology. 

The linearity of $f$ follows from that of the expectation operator, and from the convexity of $\rho$,  and that $g$ is $\rho$ composed with a linear function of $\calR$ follows the convexity of $g$. 

Weak continuity of $f$ from the fact that the coordinates of $\mathbf{R}$ over $\calR$ is a uniformly integrable set of random variables. This is so because uniform integrability, along with weak convergence, implies convergence of the first moment; see Theorem 3.5 in \citet{billingsley1999convergence}. %
%
%
By the continuous mapping theorem, the map
\begin{eqnarray*}
T:&\calR &\to\quad \calE\\
&(\X,\R) &\to \quad \sum_{i=1}^n (X_i - R_i)
\end{eqnarray*}
is weakly continuous. 
Since $\rho$ is weakly lower semi-continuous on $\calE$, $g$ is weakly lower semi-continuous on $\calR$. 

\noindent{\bf Proof of ii. and iii.:} 
Now let $\rho(0) < c <\rho \left( \sum_{i=1}^n X_i  \right)$. Since \eqref{Reins-C} is a convex programming on the convex subset $\calR$ of the real vector space  $L^1((\Omega,\calF,\mathbb{P}))^{2n}$ and Slater's condition holds since $\R=\X$ is strictly feasible, strong duality is guaranteed by Lemma \ref{bv} in the Appendix. As a result, there exists dual variable $\lambda^*\geq 0, $ such that $C(c) \subseteq P(\lambda^*)$. If $\lambda^*=0$, $P(\lambda^*)=\{\mathbf{0}\}$, contradicting $c <\rho \left( \sum_{i=1}^n X_i  \right)$. Therefore $\lambda^*>0$, and by the complementary slackness condition, every $\R^* \in C(c)$ satisfies
\[
{ g(\R,\X)}=c.
\]
Now it is easy to see for any $\lambda>0$,
\[
P(\lambda) \cap \left\{ \R: { g(\R,\X)}=c \right\}=C(c) \text{ or } \emptyset,
\] and the proof is complete.

\end{proof}

Theorem \ref{dual} provides an algorithm to solve $C(c)$, by iterating over $\lambda$ until the constraint is tight. Note that $\rho$ need not be weakly lower semi-continuous on $L^1(\Omega,\calF,\mathbb{P})$.

\section{Variance as the Risk Measure}

In this section, we consider \eqref{Reins-C} and \eqref{Reins-P} for the case when the risk measure $\rho$ equals the variance operator, {\it i.e.}  $\rho(Z)=\Var{Z}=\E{Z^2}-\left(\E{Z}\right)^2$. In Lemma \ref{var_convex_semi} of the Appendix, we establish that such a choice implies convexity and weak lower semi-continuity of $\rho$.

We note that in the original problem considered by de Finetti, with independent risks and reinsurance treaties constrained to quota share reinsurance, {\it i.e.} $R_i = a_i X_i$, the optimal proportions $a_i$ are determined to be
\[
a_i = \left(1 - \frac{\beta_i \mathbb{E}[X_i]}{2 \lambda_{\text{Fin}} \mathrm{Var}(X_i)} \right)_+
\]
with $\lambda_{\text{Fin}}$ can be solved numerically, as in \citet{de1940problema} and section 8.2.6.1 of \citet{albrecher2017reinsurance}. \citet{glineur2006finetti} provides a complete proof using the KKT conditions, and later \citet{pressacco2011mean} extended it to the group-correlation case and gave a closed-form expression. \citet{Barone2008IME} summarizes de Finetti's work and provides a geometric perspective on optimal portfolios with and without correlation. \citet{Lampaert2005SAJ} compares different kinds of proportional insurance using a numerical example.

We assume, without loss of generality, that $\beta_i$s are distinct, as otherwise we can combine the risks having the same safety loadings, solve the optimization problem, and split the reinsurance receivable to satisfy the constraints $0 \leq R_i \leq X_i$, for $i=1,\ldots,n$. Also, we assume that the safety loadings are positive, since $\beta_i=0$ implies that $X_i$ is the optimal $R_i$. 

\begin{Proposition} \label{Penalized}
When $\beta_i$s are distinct and positive then, for any $\lambda,c \geq   0$, $C(c)$ and $P(\lambda)$ are both singletons in the space $\calL^1(\Omega,\calF,\mathbb{P})$. For any $c>0$, there exists a $\lambda \geq 0$ such that \( C(c)=P(\lambda)\).
\end{Proposition} 
\begin{proof}
Without loss of generality, let $0 < \beta_1 < \cdots < \beta_n$. We note that the nontrivial case is when $c$ satisfies $0<c < \text{Var}\left(\sum_{i=1}^{n} X_i \right)$. This is so as, $C(0)=\{\X\}$, and when $c \geq \text{Var}\left(\sum_{i=1}^{n} X_i \right)$, $C(c)=P(0)=\{\mathbf{0}\}$. We consider the non-trivial case below.

 Combining Lemma \ref{var_convex_semi} of the Appendix and Theorem \ref{dual}, there exists $\lambda^* >0$ such that $C(c) \subseteq P(\lambda^*)$, and both $C(c)$ and $P(\lambda)$ are nonempty for every $\lambda \geq 0$.
Recall from the definition of $g$ in \eqref{g-def}, and variance as the risk measure,  
\[
{g(\R,\X)}= \text{Var}\left(\sum_{i=1}^{n} (X_i - R_i) \right).
\]
For an arbitrary $\lambda > 0$, since $f(\R,\X)+\lambda \cdot g(\R,\X)$ is a convex function on a convex set $\calR$, $P(\lambda)$ is a convex set. From the non strictly convex direction of $\Var{\cdot}$, there exists an $M\in L^1(\Omega)$, such that every $\R^* \in P(\lambda)$ satisfies 
\[
R^*_1+R^*_2+\dots+R^*_n = M-m \geq 0,
\]
where $m$ is a constant. It is easy to see $m={\rm ess \,inf}\, M$, and that 
\begin{equation}\label{prop1_defn_R*}
R^*_i=\left(M-m-\sum_{j=1}^{i-1} X_j\right)_+ \wedge X_i,     
\end{equation}
due to $\beta_i$'s being distinct. Note that for $\R\in P(\lambda)$ that does not satisfy the above, we will have $\sum\beta_iR_i \geq \sum\beta_iR_i^*$, with strict inequality on a set of positive measure, a contradiction. Hence, $P(\lambda)$ is a singleton for every $\lambda>0$. Hence, $C(c) \subset P(\lambda^*)$, implying $C(c)=P(\lambda^*)$ is a singleton.
\end{proof}

While it may appear natural to endow the set $\calR$ with the $L^1$ metric topology, i.e. $d((\R_1,\X),(\R_2,\X)):=\int\|\R_1-\R_2\|_1\,{\rm d}\mathbb{P}$, under this topology while we retain semi-continuity of the mean-variance disutility, $\calR$ fails to be compact. Let $D$ and $\calM$, a set of probability measures, be defined as 
\[
D=\{(x_1,x_2)\in \Real^2:0 \leq x_1 \leq x_2 < \infty\},\, \hbox{and }  
\calM=\{\eta \in \calP\left(D^n\right): \pi_{e\#}\eta=\mu\},
\]
where $\pi_e: \Real^{2n} \rightarrow \Real^n$ is the projection onto the even coordinates.    We can define the quotient map $q$ from $\calR=\{ (\R,\X):0 \leq R_i \leq X_i,  \,  i =1,\ldots,n\}$ to the set of probability measures $\calM,$ where $(\R_1,\X) \sim  (\R_2, \X)$ if and only if they have the same joint law.%
\begin{wrapfigure}{l}{0.35\textwidth}
    \begin{tikzcd}[row sep=large, column sep=3em]
        \calR \arrow[rr, "{f, g}"] \arrow[d, "q"'] & & \Real \\
        \calM \arrow[urr, dashed, "{\tilde{f}, \tilde{g}}"'] & &
    \end{tikzcd}    
\caption{Quotient Diagram}
\label{fig:wrapfig}
\end{wrapfigure}
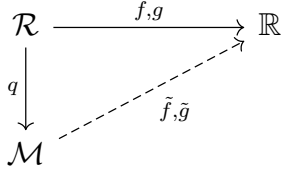%
Although $\calM$ inherits the weak topology from $\calR$ as the quotient topology, we note that the vector space structure of $\calM$ is not inherited from that of $\calR$. Besides, $\calM$ is a convex subset of $\calP\left(D^{n}\right)$, but not necessarily the image of $\calR$ under the quotient map. If $(\Omega,\calF, \mathbb{P})$ is an atomless probability space, which is equivalent to it being isomorphic to $([0,1],\calB([0,1]),\lambda|_{[0,1]})$, then the quotient map is surjective. As a result, since $f$ and $g$ are law invariant, we can define the $\tilde{f},\tilde{g}: \calM \to \Real$ such that the quotient diagram commutes. However, $\tilde{g}$ - the variance map - is concave on $\calM$, which prevents the use of standard strong duality results such as Lemma \ref{bv} of the Appendix.

The solution for \eqref{Reins-C} in \citet{acciaio2025optimal} utilizes directional derivatives of the objective and constraint functionals. If those directional derivatives can be represented as integral operators, Propositions 4.2 and 4.7 of \citet{acciaio2025optimal} can be used to work out the support of the measure of the optimal reinsurance treaties. It turns out the support is degenerate, meaning that the reinsurance amount is fully determined by the risk.

Instead, we solve the penalized problem \eqref{Reins-P}, based on variational analysis.
 The next theorem characterizes the solution set $P(\lambda)$ for $\lambda\geq 0$.

\begin{theorem}
When $0 < \beta_1 < \cdots < \beta_n$ and $\lambda \geq0$, the single element $\R^\ast$ of $P(\lambda)$ has the following functional form: for $1\leq k \leq n$,   
\begin{equation}
     \label{solOptimalreinsurance}
     R_k^{\ast} = \left(\sum_{i=k}^nX_i-\frac{\beta_k}{2\lambda}-\sigma\right)_+ \wedge X_k,
 \end{equation}
where $\sigma$ is a constant.
\end{theorem}
\begin{proof}
Since $P(0)=\{\mathbf{0}\}$, and the above specification of $R_k^\ast$ is identically zero for $\lambda=0$, we consider only the case $\lambda>0$ below. Recall that \eqref{prop1_defn_R*}, in the proof of Proposition \ref{Penalized}, contains an expression for $\R^*$. So, it suffices to show that the functional form of $\R^\ast$ in \eqref{prop1_defn_R*} coincides with that given in \eqref{solOptimalreinsurance}. 

It is easy to see,  for $i<j$, since $\beta_i<\beta_j$ and $\mathbf{X}$ is non-negative,
\begin{equation}
\label{disjoint}
    \{R_i^*<X_i\} \cap \{R_j^*>0\}=\emptyset.
\end{equation}
Further, let   
\[
A_k:=\{0<R_k^*<X_k\},\, B_k:=\{R_k^*=0\},\, \hbox{and}\, C_k:=\{R_k^*=X_k\},
\]
which, for each $1\leq k\leq n$ form a partition of $\Omega$. From \eqref{disjoint}, $R_{k+1}^*=\cdots=R_n^*=0$ on $A_k \cup B_k$. For $1\leq k \leq n$, let 
\[
Z:=\sum_{i=1}^n X_i-\sum_{i=1}^n R_i^*,\,  \sigma:=\E{Z},\, {\bar S}_k:=\sum_{i=k+1}^n X_i,\,  c_k:=\frac{\beta_k}{2 \lambda}+\sigma.
\]
Also, let $l$ denote the objective function of \eqref{Reins-P}, that is
\[
l(\R):={ f(\R,\X)}+\lambda {g(\R,\X)},
\]

We identify and exploit the problem's recursive nature. We begin with a variational analysis on $R_1$. We say $\epsilon_1\in \Real$ is admissible for a bounded non-negative random variable $T_1$ with $\E{T_1}>0$, if $0 \leq R_1^*-\epsilon_1 T_1 \leq X_1$. Since $\R^*$ is optimal, 
\[
l(R_1^*-\epsilon_1 T_1, R_2^*, \ldots, R_n^*)\geq l(\R^*),
\]
which is equivalent to 
\[
-\beta_1 \epsilon_1 \E{T_1}+\lambda\left[\epsilon_1^2\Var{T_1}+2\epsilon_1\left(\E{ZT_1}-\E{Z}\E{T_1}\right)\right] \geq 0.
\]
After some simplifications, we have 
\begin{equation}
\label{eta1}
    \epsilon_1c_1 \leq \epsilon_1 \left(\frac{\E{ZT_1}}{\E{T_1}}\right)+\frac{\epsilon_1^2 \Var{T_1}}{2\E{T_1}}=\epsilon_1 \left(\frac{\E{ZT_1}}{\E{T_1}}\right)+o(\epsilon_1).
\end{equation}

We now show that $Z=c_1$ on $A_1$. Hence, considering the case of  $\mathbb{P}(A_1)>0$, we have that ${\mathbb P}(\{\delta \leq R^*_1 \leq X_1-\delta\})>0$, for small enough $\delta>0$. For $T_1$, nonnegative, bounded, and supported on $\{\delta \leq R^*_1 \leq X_1-\delta\}$ for small enough $\delta>0$, $\epsilon_1$ in a small enough ball around $0$ are admissible. Hence,  for such $T_1$ with $\E{T_1}>0$, \eqref{eta1} yields 
\[
c_1 = \frac{\E{ZT_1}}{\E{T_1}}.
\]
This, with the flexibility in the choice of $T_1$, and since $\delta$ can be arbitrarily small, implies that $Z=c_1$    on $A_1$. Hence, on $A_1$, since $R_2^*=\ldots=R_n^*=0$,  we have $R_1^*=X_1+\bar{S}_1-c_1$. As a result, on $A_1$, $X_1+\bar{S}_1>c_1$  and $\bar{S}_1<c_1$.  
 
Now we show similarly that, $\bar{S}_1\geq c_1$ on $C_1$. Hence, in the case of  $\mathbb{P}(C_1)>0$, we consider $T_1$, nonnegative, bounded, and supported on $C_1$ with $\E{T_1}>0$ in \eqref{eta1}. For such $T_1$, $\epsilon_1>0$ in a small enough ball around $0$ are admissible. Thus, we get
\[
c_1 \leq \frac{\E{(\bar{S}_1-R_{2}^*-\cdots-R_n^*)T_1}}{\E{T_1}} \leq \frac{\E{\bar{S}_1T_1}}{\E{T_1}}. 
\]
The flexibility in the choice of $T_1$, and $R_i^*\leq X_i$ imply that $\bar{S}_1\geq c_1$ on $C_1$. 
A similar one-sided argument yields that on $B_1$, $X_1+\bar{S}_1\leq c_1$. In summary, we have this important characterization of $A_1$,$B_1$, and $C_1$:
 \[
 A_1=\{X_1+\bar{S}_1>c_1\}\cap\{\bar{S}_1<c_1\}, B_1= \{X_1+\bar{S}_1\leq c_1\}, \hbox{ and } C_1 =\{\bar{S}_1\geq c_1\}.
 \]
 Since \[R_1^*=\begin{cases}
     X_1+\bar{S}_1-c_1, &\omega\in A_1\\
     0, &\omega\in B_1\\
     X_1, &\omega\in C_1
 \end{cases},
 \]
 an alternate expression for $R_1^*$ is 
 \[
 R_1^*=(X_1+\bar{S}_1-c_1)_+ \wedge X_1=  \left(\sum_{i=1}^n X_i-c_1\right)_+ \wedge X_1.
\] 

The recursive nature arises from the fact that on $A_1 \cup B_1$, $R_2^*=\cdots=R_n^*=0$, and on $C_1$, since $R_1^*=X_1$, the sub-problem to determine $R_i^*$ for $i\geq 2$ is akin to the original but with the $n-1$ risks $X_i$, $i\geq 2$. By an analogous development as above, we have on $A_k$,
$X_k+\bar{S}_k-R_k^*=c_k$, for  $1 \leq k \leq n$, along with 
\[
A_k=\{X_k+\bar{S}_k>c_k\}\cap\{\bar{S}_k<c_k\},\, B_k= \{X_k+\bar{S}_k\leq c_k\},\, \hbox{and } C_k =\{\bar{S}_k\geq c_k\}.
\]
As a result, analogously, we have for  $1 \leq k \leq n$,
 \begin{equation}
     R_k^*=(X_k+\bar{S}_k-c_k)_+ \wedge X_k= \left(\sum_{i=k}^nX_i-\frac{\beta_k}{2\lambda}-\sigma\right)_+ \wedge X_k.
 \end{equation}
\end{proof}

\begin{Remark}
If the loadings are not distinct, {\it e.g.} $\beta_i=\beta_{i+1}=\cdots=\beta_j, j>i$ with other values distinct, then let $\R^*$ be the above solution derived by breaking ties arbitrarily. The set of all solutions to \eqref{Reins-P} in the above setting is the set of $\R$ satisfying  
\[
\sum_{k=i}^j R_k=\sum_{k=i}^jR^*_k,\; \hbox{and}\; R_k=R_k^*,\; \hbox{for}\; k<i \; \text{or}\; k>j. 
\]
Note that the lack of distinctness does not guarantee the uniqueness of the solution $\R^*$.
\end{Remark}

\noindent{\bf Algorithm to compute $\E{Z}$:} We comment on algorithms to compute $\E{Z}$, denoted by $\sigma$ in \eqref{solOptimalreinsurance}, for a fixed value of $\lambda$. Towards this, for 
\[Z_{\eta}=\sum_{i=1}^n X_i-\sum_{i=1}^n R_i^{**}=\begin{cases}
    \sum_{i=j}^n X_i,\;&\frac{\beta_{j-1}}{2\lambda}+\eta \leq\sum_{i=j}^n X_i \leq \frac{\beta_j}{2\lambda}+\eta, \;j=1,2,...,n;\\ \\
    \frac{\beta_j}{2\lambda}+\eta, \;&\sum_{i=j+1}^n X_i<\frac{\beta_j}{2\lambda}+\eta<\sum_{i=j}^n X_i,\;j=1,2,...,n,
\end{cases}\]
where $\beta_0=-\infty$, $\sum_{i=n+1}^nX_i=0$, and 
\[
R_i^{**}=\left(T_i - \eta\right)_+\wedge X_i,\, \hbox{where }  T_i:=\sum_{j=i}^n X_j \;-\;\frac{\beta_i}{2\lambda}, \, i=1,\dots,n.
\]
Since the dependence of $\E{Z_{\eta}}$ on $\eta$ arises from $\sum_{i=1}^n \E{R_i^{**}}$, and since the intervals of non-constancy of $R_i^{**}$'s,  
\[
I_i:=(T_i-X_i,\;T_i),\, i=1,\dots,n,
\]
are disjoint, $\E{Z_{\eta}}$ is non-decreasing Lipschitz with Lipschitz constant at most 1.  
We deliberately avoided computing derivatives with respect to $\eta$ here, as the probability space may not be atomless. Moreover, note that $\E{Z_{\eta}}$ is above the diagonal at $0$ and below the diagonal at $\sum_{i=1}^n \E{X_i}$. With a mild condition, such as the support of $X_n$ containing the interval 
\[
\left[0,\sum_{i=1}^n \E{X_i}+\frac{\beta_n}{2\lambda}\right],
\]
and its distribution function being continuous, it is easy to see that $\E{Z_{\eta}}$ is contractive on $[0,\sum_{i=1}^n \E{X_i}]$, and hence the unique fixed point equals $\sigma$. One could use monotonicity of  $h(\eta):=\eta-\E{Z_{\eta}}$ to compute $\sigma$ via the bisection method, or simply use fixed point iterations, with either algorithm guaranteeing a linear rate of convergence.  For the constrained version \eqref{Reins-C}, we find the $\lambda$ that attains the required variance bound via bisection or any other numerical root-finding algorithms. We illustrate the above algorithm with an example from \citet{acciaio2025optimal}.

\begin{Example}\label{Ex1}
Consider two independent random variables $X_1$ and $X_2$ with $X_1$ following a  $\hbox{Gamma}(1/2, 1/2)$ distribution and $X_2$  a (shifted) Pareto distribution with p.d.f. given by
\[
f_{X_2}(x) = 324\,(x + 3)^{-5}, \quad x \geq 0.
\]
Note that $\mathbb{E}[X_1] = \mathbb{E}[X_2] = 1$ and $\mathrm{Var}(X_1) = \mathrm{Var}(X_2) = 2$. Let $\beta_1=0.1$ and $\beta_2 = 0.25$, which reflects the fact that safety risk loadings are typically higher for heavy-tailed risks. Consider an insurer that seeks to bound the retained variance by $c = 2$. 
We simulate $10$ million samples of $(X_1,X_2)$ and using the algorithm above to calculate the parameters for \eqref{solOptimalreinsurance} to be $\sigma = 1.8029$ and $\lambda = 0.0222$. For this purpose, note that the explicit expression for $Z$, in the case of $n=2$, is given below:
\[
Z= X_1+X_2- R_1^*-R_2^*=\begin{cases}
X_1+X_2,\;&X_1+X_2 \leq \frac{\beta_1}{2\lambda}+\sigma\\ \\
\frac{\beta_1}{2\lambda}+\sigma, \;&X_2<\frac{\beta_1}{2\lambda}+\sigma<X_1+X_2\\ \\
 X_2, \;&\frac{\beta_1}{2\lambda}+\sigma \leq X_2<\frac{\beta_2}{2\lambda}+\sigma \\ \\
\frac{\beta_2}{2\lambda}+\sigma  ,\;& X_2 \geq \frac{\beta_{2}}{2\lambda}+\sigma
\end{cases}
\]
All analyses were conducted in the R software environment (version 4.3.2; see \citet{R}). See Figure \ref{fig:side_by_side} for the root finding and Figure \ref{fig:fixedpoint} for the fixed point iterations. In Figure \ref{fig:side_by_side} (a), we use {\tt uniroot} in R for root finding.

\begin{figure}[htbp]
    \centering
    \subfigure[Plot of $h(\cdot)$]{
        \includegraphics[width=0.45\textwidth]{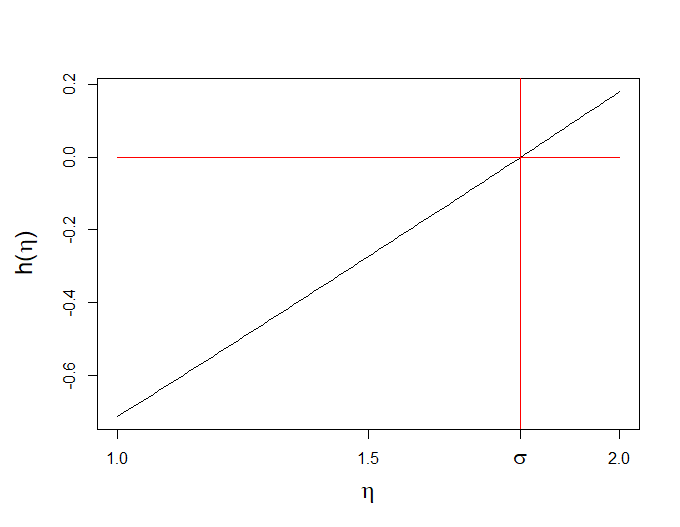}
    }
    \hspace{0.05\textwidth}
    \subfigure[Plot of objective of \eqref{Reins-P}]{
        \includegraphics[width=0.45\textwidth]{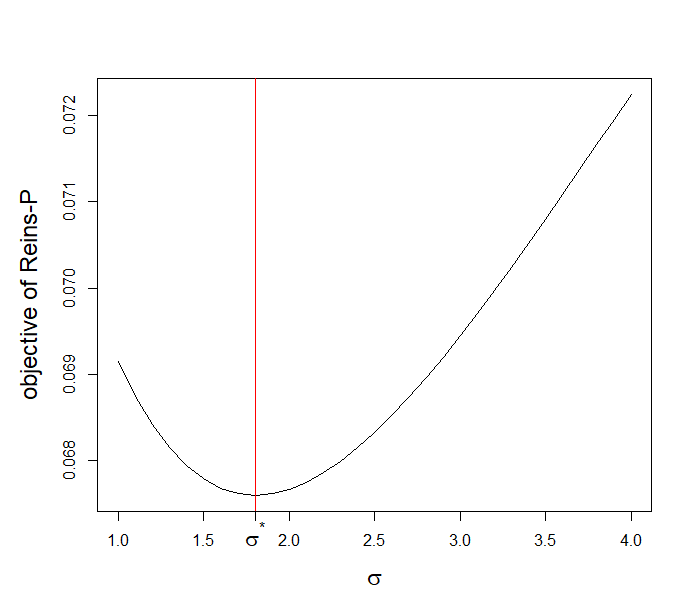}
    }
    \caption{The left picture shows $h(\eta):=\eta-\E{Z_{\eta}}$ is increasing for an absolutely continuous $\mu$, with a zero at $\sigma=1.8029$. The right picture shows the objective of \eqref{Reins-P} minimizes over all $R_i=\left(\sum_{i=k}^nX_i-\frac{\beta_k}{2\lambda^*}-\sigma\right)_+ \wedge X_k$ at this zero.}
    \label{fig:side_by_side}
\end{figure}

\begin{figure}[htbp]
    \centering
    \includegraphics[width=0.8\textwidth]{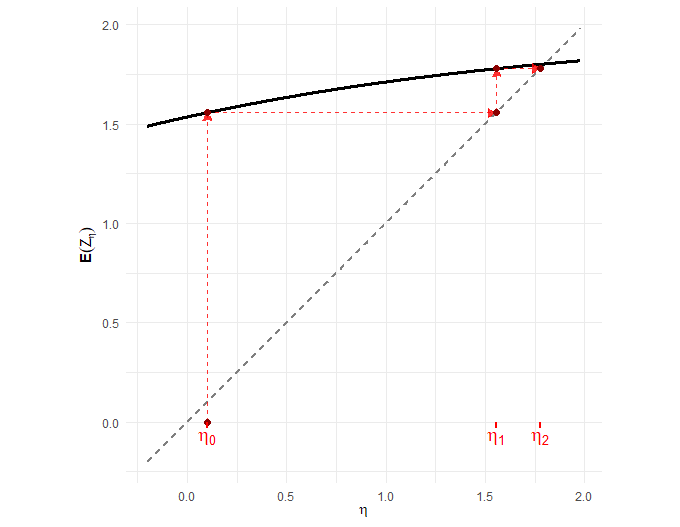}
    \caption{Fixed point iterations converge to the fixed point $\sigma=1.8029$. Initial point at $\eta_0=0.1$ followed by two iterations.}
    \label{fig:fixedpoint}
\end{figure}

\end{Example}

\begin{Remark}
While feasible reinsurance payoffs may make use of information beyond that contained in $\X$, it should not come as a surprise that optimal payoffs will be a function of $\X$. This is so, as we can always replace any reinsurance payoff $\R$ by $\E{\R|\X}$, leaving the expectation of the linear part unchanged while decreasing the variance.
\end{Remark}

\section{Conditional Value at Risk as the Risk Measure}
In this section, we let $\rho(\cdot)=\CVaR{\alpha}{\cdot}$, where the conditional value at risk (CVaR) of a random variable $X \in L^1(\Omega,\calF,\mathbb{P})$ of confidence $\alpha$ is given by,
\[
\CVaR{\alpha}{X}=\int_{\alpha}^1 \vec{F}_X(t) {\rm d}t,\,\hbox{with}\, \vec{F}_X(\alpha):=\inf\left\{ x\vert F_X(x) \geq \alpha \right\}.
\]
Note that $ \vec{F}_X$ is a generalized inverse of $F_X$, with $\vec{F}_X(\alpha)$ a definition of the \emph{Value‐at‐Risk} of $X$ at level $\alpha$, denoted as \(\VaR_\alpha(Z)\). On $L^1((\Omega,\calF,\mathbb{P}))$, the Wasserstein-$1$ metric is defined by,
\[
W_1(U,V)= \min\left\{\EP{\gamma}{|X-Y|}: \gamma \hbox{ is a coupling of } \calL(U) \hbox{ and } \calL(V)\right\}.
\]
Since $\calE$ is uniformly integrable, from Proposition 7.1.5 of \citet{luigi2008gradient}, we have that on $\calE$, the weak topology is metrized by $W_1$. We have the following lemma,
\begin{lemma}
\label{cvar_convex_lsc}
    $\CVaR{\alpha}{\cdot} : \calE \to \Real$ is convex, weakly continuous, and moreover is 1-Lipschitz with respect to $W_1$.
\end{lemma}

\begin{proof}
CVaR is a coherent risk measure, see \citet{artzner1999coherent}, and hence it is both subadditive and positively homogeneous. These two properties imply that CVaR is a convex risk measure.  Now, since
\begin{align*}
    (1-\alpha)\left| \text{CVaR}_\alpha(U) - \text{CVaR}_\alpha(V) \right| &= \left|  \int_{\alpha}^{1} \vec{F}_U(u) \, du -  \int_{\alpha}^{1} \vec{F}_V(u) \, du \right|\\
    &\le  \int_{\alpha}^{1} \left| \vec{F}_U(u) - \vec{F}_V(u) \right| \, du\, \\
    &\le  \int_{0}^{1} \left| \vec{F}_U(u) - \vec{F}_V(u) \right| \, du\, =  W_1(U, V),
\end{align*}
we have CVaR is 1-Lispchitz with respect to $W_1$, which implies weak continuity of CVaR.
\end{proof}
From Lemma \ref{cvar_convex_lsc} and Theorem \ref{dual}, for every $0<c<\text{CVaR}_\alpha \left( \sum_{i=1}^n X_i \right)$, $C(c)$ is nonempty, and there exists $\lambda^*>0$ such that $C(c) \subset P(\lambda^*)$. 
The following theorem characterize the solution set $P(\lambda)$.

\newpage
\begin{theorem}
\label{P(lambda)_for_cvar}
    Let $K:=\frac{\lambda}{1-\alpha}$, and $S := \sum X_i$ be the total loss. For any $\lambda>0$, $P(\lambda)$ is the set of $\R^*$ satisfying the following properties,
    \begin{enumerate}
        \item If $\beta_i \ne K$, 
        \[
R_i^*=\mathbf{I}(\beta_i<K) \left(\left(S-\sum_{j=1}^{i-1}X_j-q \right)_+ \wedge X_i\right).
\]
        \item If $\beta_i = K$, we have pointwise
        \[
        0 \leq R_i^* \leq \left(S-\sum_{j=1}^{i-1}X_j-q \right)_+ \wedge X_i,
        \]
    \end{enumerate}
    where $q\in[0,\text{VaR}_\alpha(S)]$ satisfies the following
\begin{equation}\label{con:q}
\begin{cases}
J_-'(q) \le 0 \le J_+'(q), & \hbox{if } q \in (0,\text{VaR}_\alpha(S)) \\[6pt]
J_+'(q) \ge 0, & \hbox{if } q = 0 \\[6pt]
J_-'(q) \le 0, & \hbox{if } q =\text{VaR}_\alpha(S)
\end{cases},    
\end{equation}
where, defining $\beta_0:=0$,
\[
J_-'(q):=\sum_{i=1}^n \left((K-\beta_i)_+-(K-\beta_{i-1})_+ \right) \P{q \leq \sum_{j=i}^n X_j} + \lambda,
\]
and
\[
J_+'(q):=\sum_{i=1}^n \left((K-\beta_i)_+-(K-\beta_{i-1})_+ \right) \P{q < \sum_{j=i}^n X_j} + \lambda
\]
\end{theorem}

\begin{Remark}
    If the laws of the tail partial sums of $X$ are atomless, which is true if the law $\mu$ of $\X$ is {absolutely continuous with respect to the Lesbesgue measure}, $P(\lambda)$ will be a singleton for all $\lambda \in [0,\infty)\backslash\{(1-\alpha)\beta_i: 1\leq i\leq n\}$. 
\end{Remark}

\begin{Remark}
Neither $C(c)$ nor $P(\lambda)$ may be a singleton; in the following single risk toy example, we show that $C(c)$ is not a singleton. Let $(\Omega,\calF,\mathbb{P})=([0,1],\calB([0,1]),\mu)$, where $\calB([0,1])$ stands for the Borel sets on $[0,1]$, and $\mu$ is the Lebesgue measure. Let $X\sim U[0,1]$ be the identity map. We set $\alpha=0.9$, and $c=0.92$. Note that any feasible $R$ satisfies,  
\begin{multline*}
    0.92 \geq \CVaR{0.9}{X-R} \geq \E{X-R \big|X \geq 0.9}
    =0.95-\frac{\E{R \mathbf{1} (X \geq 0.9)}}{0.1} \geq 0.95-\frac{\E{R}}{0.1},
\end{multline*}
implying $\E{R}\geq 0.3\%$. It is easy to see that for any $R \leq (X-0.9)_+$ with $\E{R}=0.3\%$ all the inequalities above reduce to an equality, and hence any such $\R$ is optimal.  

\end{Remark}
\begin{proof}
Using the variational representation for CVaR (see \citet{rockafellar2000optimization}), \eqref{Reins-P} is equivalent to minimizing over both $R$ and the quantile variable $q$:

\begin{multline}
\label{cvar_variational}
    \min_{\R} \quad \sum_{i=1}^n \beta_i \E{R_i}+\lambda \text{CVaR}_\alpha \left( \sum_{i=1}^n (X_i - R_i) \right)\\
    =\min_{\R, q} \quad \sum_{i=1}^n \beta_i \mathbb{E}[R_i] + \lambda \left( q + \frac{1}{1-\alpha} \mathbb{E}\left[ \left( \sum_{i=1}^n (X_i - R_i) - q \right)_+ \right] \right)
\end{multline}

For a fixed $q$, we seek to minimize the term inside the expectation for every state $\omega$ (pointwise optimization):

\[
\min_{0 \le R_i(\omega) \le X_i(\omega)} \quad \sum_{i=1}^n \beta_i R_i(\omega) + \frac{\lambda}{1-\alpha} \left( S(\omega) - \sum_{i=1}^n R_i(\omega) - q \right)_+ + \lambda q
\]

Recall $K=\frac{\lambda}{1-\alpha}$. If $K \ne \beta_i$, the solution is 
\[
R_i^*=\mathbf{I}(\beta_i<K) \left(\left(S-\sum_{j=1}^{i-1}X_j-q \right)_+ \wedge X_i\right),
\]
non-increasing in $q$; otherwise, $R_i^*$ can be any non-negative random variable bounded above by 
\[
\left(S-\sum_{j=1}^{i-1}X_j-q \right)_+ \wedge X_i.
\]
The question now is how to determine the optimal value for $q$, say $q^*$. 
For 
\[
J(q):=\min_{\R} \quad \sum_{i=1}^n \beta_i \mathbb{E}[R_i] + \lambda \left( q + \frac{1}{1-\alpha} \mathbb{E}\left[ \left( \sum_{i=1}^n (X_i - R_i) - q \right)_+ \right] \right),
\]
as a partial minimization of a jointly convex function, it is convex in $q$ (see section 3.2.6 in \citet{boyd2004convex}). By the dominated convergence theorem, the right derivative in $q$ is 
\begin{align*}
     J_+'(q)&=\sum_{i=1}^n -\beta_i \mathbf{I}(\beta_i<K) \E{\mathbf{1}\left(\sum_{j=i+1}^n X_j \le q<\sum_{j=i}^n X_j \right)} + \lambda\\
     &\qquad -K   \E{ \sum_{i=1}^n \mathbf{I}\left(\beta_i \geq K, {\sum_{j=i+1}^n X_j \le q<\sum_{j=i}^n X_j}\right) }\\
     &=\sum_{i=1}^n -\beta_i \mathbf{I}(\beta_i<K) \P{\sum_{j=i+1}^n X_j \le q<\sum_{j=i}^n X_j} + \lambda \\
     & \qquad -K \sum_{i=1}^n  \mathbf{I}(\beta_i \geq K) \P{\sum_{j=i+1}^n X_j \le q<\sum_{j=i}^n X_j}\\
     &=\sum_{i=1}^n (K-\beta_i) \mathbf{I}(\beta_i<K) \P{\sum_{j=i+1}^n X_j \le q<\sum_{j=i}^n X_j} + \lambda-K\P{q < S}\\
     &= \sum_{i=1}^n \left((K-\beta_i)_+-(K-\beta_{i-1})_+ \right) \P{q<\sum_{j=i}^n X_j} + \lambda,
\end{align*}
where $\beta_0:=0$, and $\sum_{j=n+1}^nX_j:=0$. Similarly, the left derivative in q is
\[
J_-'(q)=\sum_{i=1}^n \left((K-\beta_i)_+-(K-\beta_{i-1})_+ \right) \P{q \leq \sum_{j=i}^n X_j} + \lambda
\]
Since $f(\cdot)$ is convex, and it is easy to check that the optimal value for $q$, say $q^*$,  belongs to $[0,\text{VaR}_\alpha(S)]$, we have $q^*$ satisfying the constraints in \eqref{con:q}. 
Note that since the function is convex, a bisection algorithm can be used to find $q^*$ .

\end{proof}

Below, we provide an illustrative example with the same risk distributions as Example \ref{Ex1} of Section 3.

\begin{Example} \label{Ex2}
Consider two independent random variables $X_1$ and $X_2$ with $X_1$ following a  $\hbox{Gamma}(1/2, 1/2)$ distribution and $X_2$  a (shifted) Pareto distribution with p.d.f. given by
\[
f_{X_2}(x) = 324\,(x + 3)^{-5}, \quad x \geq 0,
\]
with a mean equal to 1 for both distributions. Let $\alpha=0.9$, and $\beta_1=0.1$ and $\beta_2 = 0.25$. Consider an insurer that seeks to bound the retained CVaR by $c = 5$. 
We simulate $10$ million samples of $(X_1,X_2)$ and estimate $\VaR_{0.9}(X_1+X_2)$ and $\CVaR{0.9}{X_1+X_2}$ to be $4.3867$ and $6.5315$, respectively. We find that $\lambda^*=0.0106$, and $q^*=4.3079$. If $c=6$ instead, then $\lambda^*=0.01$, corresponding to the discontinuity of $\CVaR{\alpha}{Z^*}$, as shown in the top left plot of Figure \ref{fig:cvar}. Solving for $q^*$ using \eqref{con:q} we see that $q^*=4.3867$, equal to  $\VaR_{0.9} (X_1+X_2)$. The solution set is
\[
C(c)=\{(R_1^*,0): 0 \leq R_1^* \leq \left(X_1+X_2-4.3867 \right)_+ \wedge X_1,\CVaR{\alpha}{X_1+X_2-R_1^*}=6\}.
\] 
Figure \ref{fig:deriv_q} shows the process of finding $q^*$ under $c=5$.
Again, all analyses were conducted in the R software environment (version 4.3.2; see \citet{R}), with the workhorse function being {\tt uniroot}.
\end{Example}

\begin{Remark}
Figure \ref{fig:cvar} depicts interesting phenomena. With $Z^*=\sum X_i-\sum R_i^*$, that $q^*=\VaR_{\alpha}(Z^*)$ solving \eqref{con:q} is continuous with respect to $\lambda$  for absolutely continuous $\mu$ is shown in the plot to the right. In contrast, $\CVaR{\alpha}{Z^*}$ as a function of $\lambda$ has discontinuities at ${\beta_1}/{(1-\alpha)}$ and ${\beta_2}/{(1-\alpha)}$, as can be seen in the expression of Theorem \ref{P(lambda)_for_cvar} and shown in the plot to the left. 
\end{Remark}

\begin{figure}[htbp]
    \centering
    \input{obj_deriv_q}
    \caption{The plot shows the zero of $J_+'(q)$ corresponding to $\lambda^*=1.06\%$ under $c=5$}
    \label{fig:deriv_q}
\end{figure}
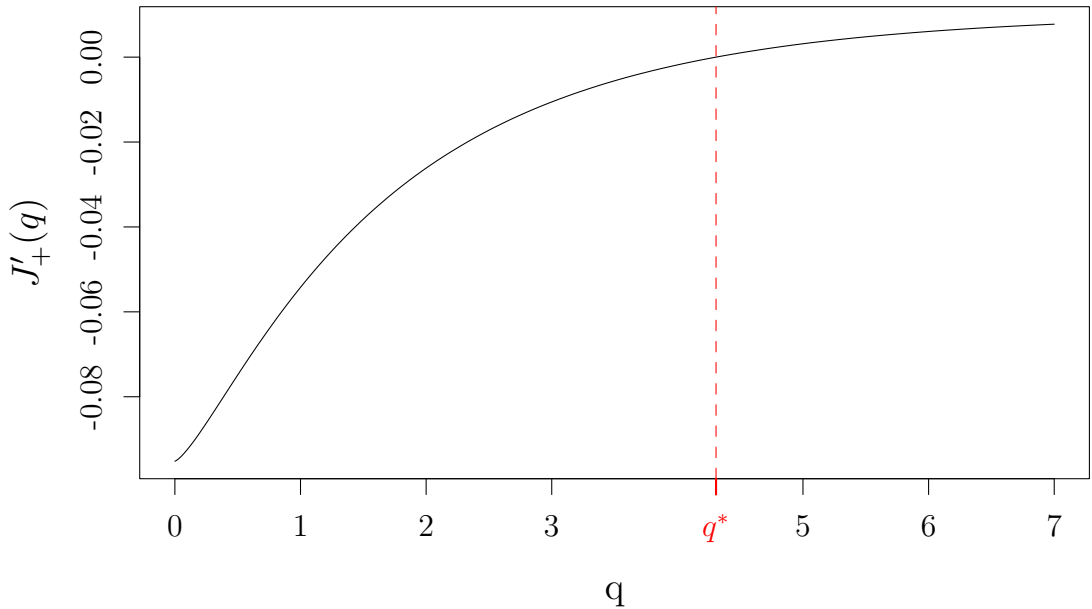

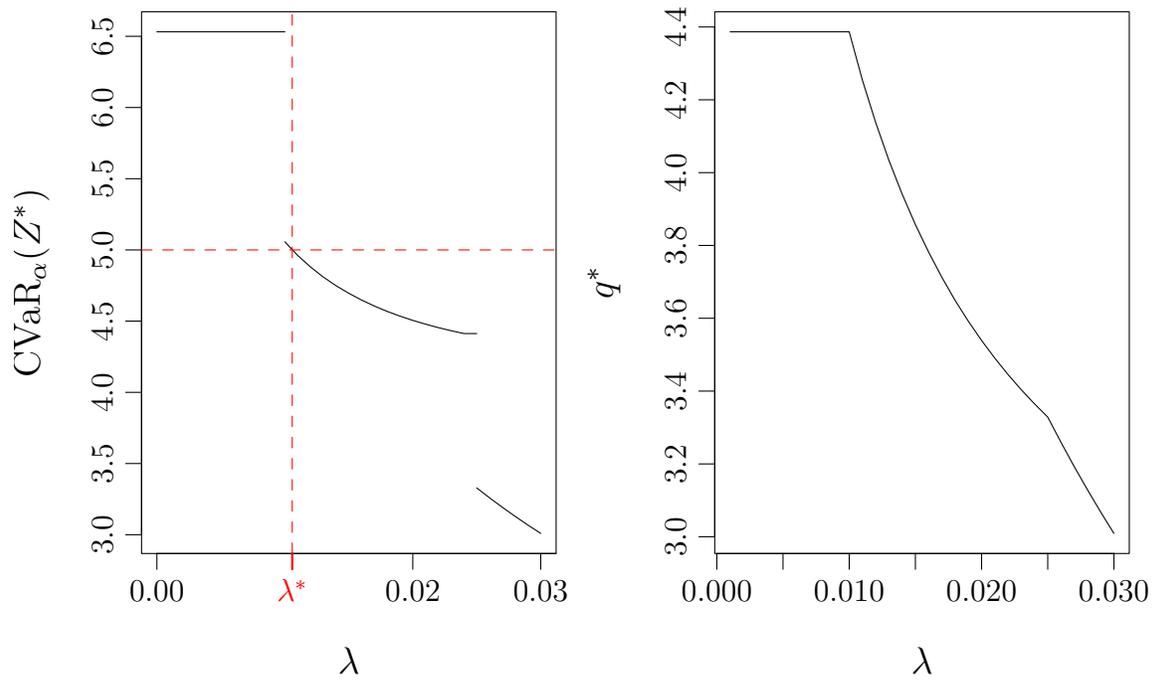
\begin{figure}[htbp]
    \centering
    \input{cvar}
    \caption{Left plot: $\lambda^*$ corresponding to the constraint $\CVaR{\alpha}{Z^*}=c$
    on the optimal retained risk. The red dashed lines correspond to the $c=5$ case. We notice that discontinuities at $0.01$ and $0.025$ corresponding to ${\beta_1}/{(1-\alpha)}$ and $\beta_2/(1-\alpha)$, respectively. Right plot:  Solution of \eqref{con:q}, $q^*$, is a continuous function of $\lambda$.}
    \label{fig:cvar}
\end{figure}

\section{Conclusion}
We provide methods to solve a class of constrained optimal reinsurance problems via convex duality. We provide specific solutions for the risk measures variance and CVaR, both convex and lower semicontinuous on a subspace of $L^1$. Our approach preserves the convexity of the constraint set, since we are working with random variables, whereas in Example 6.1 of \citet{acciaio2025optimal} the constraint set is not convex as a set of measures. The proof of Example 6.1  of \citet{acciaio2025optimal} relies on a descent direction $\eta$ of the constraint functional at the optimal insurance contract. In particular, the descent condition requires a
\[
\eta \in \calP\left(D^n\right),\, \hbox{with }  D=\{(x_1,x_2)\in \Real^2:0 \leq x_1 \leq x_2 < \infty\},
\]
satisfying $ \pi_{e\#}\eta=\mu$ where $\pi_e: \Real^{2n} \rightarrow \Real^n$ is the projection onto the even coordinates, such that 
\[
\calG(\eta^*)+d\calG(\eta^*;\eta-\eta^*)<0,
\]
where 
\[
\calG(\eta^*)=\int \left(\sum_{i=1}^n (x_{2i}-x_{2i-1})\right)^2 d\eta^*- \left(\int \left(\sum_{i=1}^n (x_{2i}-x_{2i-1})\right) d\eta^* \right)^2,
\]
the variance functional on the retained risk, $d\calG $ denote the directional derivative and $\eta^*$ denote the optimal reinsurance policy. After simplification,  we have 
\[
\calG(\eta^*)+d\calG(\eta^*;\eta-\eta^*)=\sigma^2-c+\int\left(\sum_{i=1}^n (x_{2i}-x_{2i-1})\right)^2-2\sigma \left(\sum_{i=1}^n (x_{2i}-x_{2i-1})\right) d\eta,
\] 
where $\sigma=\EP{\eta^*}{\sum_{i=1}^n X_{2i}-X_{2i-1}}$. Note that under full reinsurance, the above equals $\sigma^2-c$, which may not be negative; hence, the full reinsurance may not be a descent direction as claimed in \citet{acciaio2025optimal}. The issue stems from the non-convexity of the $\calG$. Instead of the full reinsurance, a descent direction is one minimizing the above,  which is any probability measure supported on a set $S$ satisfying
\begin{eqnarray*}
&S\cap\left\{\sum_{i=1}^n x_{2i}>\sigma\right\}\!\!\!&\subseteq \left\{ \sum_{i=1}^n x_{2i-1}=\sum_{i=1}^n x_{2i}-\sigma\right\}, \\
\hbox{and } &S\cap\left\{\sum_{i=1}^n x_{2i}\leq\sigma\right\}\!\!\!&\subseteq \left\{ x_{2i-1}=0,\, i=1,2,\cdots,n\right\}.
\end{eqnarray*}

Finally, we note that for some non-convex risk measures \eqref{Reins-C} can still be explicitly solved. One such is $\rho(\cdot)=\VaR_\alpha ({\cdot})$, the value at risk, which is considered in  Example 6.3 of \citet{acciaio2025optimal}, under the additional assumption of the absolute continuity of the law $\mu$ of $\X$. The rudimentary solution we provide below does not require the absolute continuity assumption. With $Z=\sum X_i-\sum R_i$ denoting the retained risk, the constrained version of the resulting optimization problem 
\begin{equation}\label{eq:reins_VaR}
\begin{aligned}
    C(c)=\argmin_{\R}\;
&\sum_{i=1}^n \beta_i\,\EY[R_i] \\
\text{with } & 0\le R_i\le X_i, \; \forall i\\
&\VaR_{\alpha}(Z) \leq c
\end{aligned}
\end{equation}
where $\alpha\in(0,1)$, and $0<\beta_1<\beta_2<\cdots<\beta_n$ are given parameters, and $\boldsymbol\beta=(\beta_1,\ldots,\beta_n)$. We show below that $C(c)$ is nonempty. Towards doing so, we define the map $\phi:\Real^n \to \Real^n$,
\begin{equation*}
    \phi(\mathbf{x})= \left(\left(\sum_{j=1}^{n}x_j-c \right)_+ \wedge x_1,\;\ldots, \left(\sum_{j=n-1}^{n}x_j-c \right)_+ \wedge x_{n-1}, \; \left(x_n-c \right)_+\right)'.
\end{equation*}
For any feasible $\R$, note that $\R'$ defined as 
\[
\R':=\begin{cases}
\mathbf{0}, & \sum_{i=1}^n R_i<\sum_{i=1}^n X_i - c;\\
\phi(\mathbf{X}),& \hbox{otherwise}; 
\end{cases},
\]
is a feasible reinsurance payoff with a lower objective value. Let $q=\VaR_\alpha(\boldsymbol\beta \,\phi(\mathbf{X}))$. It is relatively straightforward now to see that if $\Pr{\boldsymbol\beta \,\phi(\mathbf{X})<q}=\alpha$, $C(c)$ is a singleton containing 
\[
\R^*=\begin{cases}
\mathbf{0}, &\boldsymbol\beta \,\phi(\mathbf{X}) \geq q\\
    \phi(\mathbf{X}), &\boldsymbol\beta \,\phi(\mathbf{X})<q
\end{cases},
\] 
and that If $\Pr{\boldsymbol\beta \,\phi(\mathbf{X})<q}<\alpha$ and the probability space is atomless, $C(c)$, is no longer a singleton, contains $\R^*$ satisfying  
\[
\R^*=\begin{cases}
    \mathbf{0}, &\boldsymbol\beta \,\phi(\mathbf{X}) > q \\
    \phi(\X), &\boldsymbol\beta \,\phi(\mathbf{X})<q \\
    B\cdot \phi(\X), &\boldsymbol\beta \,\phi(\mathbf{X})=q
\end{cases},
\]
where $B\in\{0,1\}$ is a random variable satisfying
\[
\int_{\boldsymbol\beta \,\phi(\mathbf{X})=q} B \,{\rm d}\mathbb{P}=\alpha-\Pr{\boldsymbol\beta \,\phi(\mathbf{X})<q}. 
\]
Existence of such a $B$ is guaranteed by the atomless nature of the probability space.

\bibliographystyle{plainnat}
\bibliography{ref}

\section{Appendix}

\begin{lemma}
\label{var_convex_semi}
The variance operator on $\calE$ is convex and weakly lower semi-continuous.
\end{lemma}

\begin{proof}
For any $Z,Y \in \calE \subset L^1(\Omega,\calF,\mathbb{P})$,
\[
\Var{\frac{Z+Y}{2}} \leq \frac{\Var{Z}+\Var{Y}}{2},
\] 
with equality only when $\mathbb{P}$ almost surely $Z=Y+m$ for a constant $m$.
We already proved the expectation operator $\E{\cdot}$ is weakly continuous on $\calE$, as a result of the uniform integrability of $\calE$. Since $\psi(t)=t^2$ is continuous and non-negative, $g(Z)=\E{Z^2}$ is weakly lower semi-continuous. Thus, $\Var{Z}=\E{Z^2}-\left(\E{Z}\right)^2$ implies that the variance operator is weakly lower semi-continuous.
\end{proof}

    \begin{lemma}
    \label{bv}
        Let $\calR$ be a convex set in a real vector space, and $\phi,\psi:\calR \to \Real$ be two convex functions. If there exists a point $x \in \calR$ satisfying $\psi(x)<0$, which is called Slater's condition, we have strong duality:  \[p^*:=\inf_{\psi(\R) \leq 0, \R \in \calR}\phi(\R)=\sup_{\lambda \geq 0}\inf_{\R \in \calR}\phi(\R)+\lambda \cdot \psi(\R)=:d^*.\] Moreover, the optimum $\lambda^*$ attaining $d^*$ exists.
    \end{lemma}
    
    \begin{proof}
    The following proof is an adaptation of one from \citet{boyd2004convex} to our setting.
        For every $\lambda \geq 0$, $$\inf_{\psi(\R) \leq 0, \R \in \calR}\phi(\R) \geq \inf_{\psi(\R) \leq 0, \R \in \calR}\phi(\R)+\lambda \cdot \psi(\R) \geq \inf_{\R \in \calR}\phi(\R)+\lambda \cdot \psi(\R).$$
        Taking supremum on $\lambda$, we have the weak duality $p^* \geq d^*$.

        To show the strong duality, we define $\calA=\{(u,t)\in \Real^2: \exists \R \in \calR, u \geq \psi(\R), t \geq \phi(\R)\}$ and $\calB=\{(0,s)\in \Real^2: s<p^*\}$ to be two convex sets. We have $$p^*=\inf_{(0,t)\in \calA}t, \quad d^*=\sup_{\lambda \geq 0}\inf_{(u,t) \in \calA} (\lambda,1)^T(u,t).$$  It is easy to see that
        $\calA \cap \calB=\emptyset$. By the separating hyperplane theorem, 
there exists \((\tilde\lambda,\mu)\neq 0\) and \(\alpha\) such that
\begin{align}
  (u,t)\in \calA
  &\;\Longrightarrow\;
  \tilde\lambda u  + \mu\,t \;\ge\;\alpha,
  \label{5.39}\\
  (u,t)\in B
  &\;\Longrightarrow\;
  \tilde\lambda u + \mu\,t \;\le\;\alpha.
  \label{5.40}
\end{align}
From \eqref{5.39} we conclude that \(\tilde\lambda \geq 0\) and \(\mu\ge0\).  (Otherwise \(\tilde\lambda u+\mu\,t\) is unbounded below over \(\calA\), contradicting \eqref{5.39}.)  The condition \eqref{5.40} simply means \(\mu\,t\leq \alpha\) for all \(t<p^*\), and hence \(\mu\,p^*\le\alpha\).  Together with \eqref{5.39} we conclude that for any \(\R\in \calR\),
\begin{equation}\label{5.41}
  \tilde\lambda\,\psi(\R)\;+\;\mu\,\phi(\R)
  \;\ge\;\alpha\;\ge\;\mu\,p^*.
\end{equation}
If \(\mu=0\), we take $\R$ such that $\psi(\R)<0$ to conclude $\tilde \lambda=0$, which is a contradiction. Therefore, \(\mu>0\), and we can divide \eqref{5.41} by \(\mu\) to obtain
\[
   \lambda^*\,\psi(\R)\,+\,\phi(\R)
  \;\ge\;p^*
  \quad\text{for all } \R\in \calR,
\]
where \(\lambda^*=\tilde\lambda/\mu\).
Combined with weak duality, we have \(p^*=d^*\), and the dual optimum $\lambda^*$ is attained.
\end{proof}

\end{document}

%% file: math.tex
\def \R{\mathbf{R}}

\def\1{\mathbf{1}}

\def\X{\mathbf{X}}

\def\VaR{{\rm VaR}}
\def\CVaR#1#2{{\rm CVaR}_{#1}\left(#2\right)}

\DeclareMathOperator*{\argmin}{argmin}

\def\E#1{\mathbb{E}{\left(#1\right)}}
\def\EP#1#2{\mathbb{E}_{#1}{\left(#2\right)}}
\def\Var#1{{\rm Var}\left(#1\right)}

\def\Pr#1{{\rm Pr}\left(#1\right)}

\makeatletter
\DeclareRobustCommand{\cev}[1]{%
  {\mathpalette\do@cev{#1}}%
}
\newcommand{\do@cev}[2]{%
  \vbox{\offinterlineskip
    \sbox\z@{$\m@th#1 x$}%
    \ialign{##\cr
      \hidewidth\reflectbox{$\m@th#1\vec{}\mkern4mu$}\hidewidth\cr
      \noalign{\kern-\ht\z@}
      $\m@th#1#2$\cr
    }%
  }%
}
\makeatother

\usepackage{stackengine}
\stackMath
\setstackgap{L}{.4ex}

\def\calR{{\mathcal{R}}}
\def\calL{{\mathcal{L}}}
\def\calF{{\mathcal{F}}}

\def\calA{{\mathcal{A}}}

\def\calC{{\mathcal{C}}}

\def\calB{{\mathcal{B}}}
\def\calG{{\mathcal{G}}}
\def\calM{{\mathcal{M}}}

\def\calP{{\mathcal{P}}}

\def\calE{{\mathcal{E}}}

\def\Real{\mathbb{R}}

\DeclareMathOperator{\EY}{\mathbb{E}}

\def\E#1{\mathbb{E}{\left(#1\right)}}
\def\EP#1#2{\mathbb{E}_{#1}{\left(#2\right)}}
\def\Var#1{{\rm Var}\left(#1\right)}

\def\Pr#1{{\rm Pr}\left(#1\right)}
\def\P#1{{\mathbb{P}}\left(#1\right)}

%% file: obj_deriv_q.tex
\begin{tikzpicture}[x=1pt,y=1pt]
\definecolor{fillColor}{RGB}{255,255,255}
\path[use as bounding box,fill=fillColor,fill opacity=0.00] (0,0) rectangle (433.62,289.08);
\begin{scope}
\path[clip] ( 49.20, 61.20) rectangle (408.42,239.88);
\definecolor{drawColor}{RGB}{0,0,0}

\path[draw=drawColor,line width= 0.4pt,line join=round,line cap=round] ( 62.50, 67.82) --
	( 62.98, 68.05) --
	( 63.45, 68.38) --
	( 63.93, 68.76) --
	( 64.41, 69.18) --
	( 64.88, 69.64) --
	( 65.36, 70.13) --
	( 65.83, 70.65) --
	( 66.31, 71.19) --
	( 66.78, 71.75) --
	( 67.26, 72.32) --
	( 67.73, 72.91) --
	( 68.21, 73.51) --
	( 68.68, 74.12) --
	( 69.16, 74.75) --
	( 69.63, 75.39) --
	( 70.11, 76.04) --
	( 70.58, 76.70) --
	( 71.06, 77.36) --
	( 71.53, 78.04) --
	( 72.01, 78.73) --
	( 72.48, 79.42) --
	( 72.96, 80.12) --
	( 73.43, 80.82) --
	( 73.91, 81.53) --
	( 74.38, 82.23) --
	( 74.86, 82.94) --
	( 75.33, 83.66) --
	( 75.81, 84.38) --
	( 76.28, 85.11) --
	( 76.76, 85.83) --
	( 77.23, 86.56) --
	( 77.71, 87.29) --
	( 78.18, 88.01) --
	( 78.66, 88.75) --
	( 79.13, 89.48) --
	( 79.61, 90.21) --
	( 80.09, 90.95) --
	( 80.56, 91.68) --
	( 81.04, 92.42) --
	( 81.51, 93.16) --
	( 81.99, 93.89) --
	( 82.46, 94.62) --
	( 82.94, 95.35) --
	( 83.41, 96.09) --
	( 83.89, 96.82) --
	( 84.36, 97.55) --
	( 84.84, 98.28) --
	( 85.31, 99.01) --
	( 85.79, 99.74) --
	( 86.26,100.46) --
	( 86.74,101.18) --
	( 87.21,101.91) --
	( 87.69,102.63) --
	( 88.16,103.35) --
	( 88.64,104.07) --
	( 89.11,104.79) --
	( 89.59,105.50) --
	( 90.06,106.21) --
	( 90.54,106.92) --
	( 91.01,107.63) --
	( 91.49,108.33) --
	( 91.96,109.02) --
	( 92.44,109.73) --
	( 92.91,110.43) --
	( 93.39,111.12) --
	( 93.86,111.81) --
	( 94.34,112.50) --
	( 94.82,113.18) --
	( 95.29,113.86) --
	( 95.77,114.54) --
	( 96.24,115.22) --
	( 96.72,115.89) --
	( 97.19,116.57) --
	( 97.67,117.24) --
	( 98.14,117.91) --
	( 98.62,118.57) --
	( 99.09,119.24) --
	( 99.57,119.90) --
	(100.04,120.56) --
	(100.52,121.21) --
	(100.99,121.87) --
	(101.47,122.52) --
	(101.94,123.16) --
	(102.42,123.80) --
	(102.89,124.44) --
	(103.37,125.07) --
	(103.84,125.70) --
	(104.32,126.33) --
	(104.79,126.96) --
	(105.27,127.58) --
	(105.74,128.20) --
	(106.22,128.82) --
	(106.69,129.43) --
	(107.17,130.04) --
	(107.64,130.64) --
	(108.12,131.25) --
	(108.59,131.85) --
	(109.07,132.45) --
	(109.55,133.05) --
	(110.02,133.64) --
	(110.50,134.23) --
	(110.97,134.82) --
	(111.45,135.40) --
	(111.92,135.98) --
	(112.40,136.56) --
	(112.87,137.13) --
	(113.35,137.70) --
	(113.82,138.27) --
	(114.30,138.83) --
	(114.77,139.39) --
	(115.25,139.96) --
	(115.72,140.51) --
	(116.20,141.06) --
	(116.67,141.61) --
	(117.15,142.15) --
	(117.62,142.70) --
	(118.10,143.24) --
	(118.57,143.78) --
	(119.05,144.31) --
	(119.52,144.85) --
	(120.00,145.37) --
	(120.47,145.90) --
	(120.95,146.42) --
	(121.42,146.94) --
	(121.90,147.46) --
	(122.37,147.98) --
	(122.85,148.49) --
	(123.32,149.00) --
	(123.80,149.50) --
	(124.28,150.00) --
	(124.75,150.50) --
	(125.23,150.99) --
	(125.70,151.49) --
	(126.18,151.98) --
	(126.65,152.47) --
	(127.13,152.96) --
	(127.60,153.44) --
	(128.08,153.92) --
	(128.55,154.40) --
	(129.03,154.87) --
	(129.50,155.34) --
	(129.98,155.81) --
	(130.45,156.27) --
	(130.93,156.73) --
	(131.40,157.20) --
	(131.88,157.65) --
	(132.35,158.11) --
	(132.83,158.57) --
	(133.30,159.02) --
	(133.78,159.46) --
	(134.25,159.90) --
	(134.73,160.35) --
	(135.20,160.79) --
	(135.68,161.23) --
	(136.15,161.66) --
	(136.63,162.09) --
	(137.10,162.52) --
	(137.58,162.94) --
	(138.05,163.36) --
	(138.53,163.78) --
	(139.00,164.20) --
	(139.48,164.62) --
	(139.96,165.03) --
	(140.43,165.44) --
	(140.91,165.85) --
	(141.38,166.26) --
	(141.86,166.67) --
	(142.33,167.07) --
	(142.81,167.47) --
	(143.28,167.87) --
	(143.76,168.27) --
	(144.23,168.66) --
	(144.71,169.05) --
	(145.18,169.44) --
	(145.66,169.83) --
	(146.13,170.21) --
	(146.61,170.60) --
	(147.08,170.98) --
	(147.56,171.35) --
	(148.03,171.73) --
	(148.51,172.10) --
	(148.98,172.46) --
	(149.46,172.84) --
	(149.93,173.20) --
	(150.41,173.56) --
	(150.88,173.92) --
	(151.36,174.29) --
	(151.83,174.64) --
	(152.31,175.00) --
	(152.78,175.36) --
	(153.26,175.71) --
	(153.73,176.06) --
	(154.21,176.40) --
	(154.69,176.74) --
	(155.16,177.09) --
	(155.64,177.43) --
	(156.11,177.77) --
	(156.59,178.11) --
	(157.06,178.45) --
	(157.54,178.78) --
	(158.01,179.11) --
	(158.49,179.44) --
	(158.96,179.76) --
	(159.44,180.08) --
	(159.91,180.41) --
	(160.39,180.73) --
	(160.86,181.05) --
	(161.34,181.36) --
	(161.81,181.68) --
	(162.29,181.99) --
	(162.76,182.30) --
	(163.24,182.61) --
	(163.71,182.92) --
	(164.19,183.23) --
	(164.66,183.53) --
	(165.14,183.83) --
	(165.61,184.13) --
	(166.09,184.43) --
	(166.56,184.72) --
	(167.04,185.02) --
	(167.51,185.31) --
	(167.99,185.60) --
	(168.46,185.89) --
	(168.94,186.18) --
	(169.42,186.46) --
	(169.89,186.75) --
	(170.37,187.04) --
	(170.84,187.32) --
	(171.32,187.60) --
	(171.79,187.88) --
	(172.27,188.16) --
	(172.74,188.44) --
	(173.22,188.71) --
	(173.69,188.98) --
	(174.17,189.25) --
	(174.64,189.52) --
	(175.12,189.79) --
	(175.59,190.06) --
	(176.07,190.32) --
	(176.54,190.58) --
	(177.02,190.84) --
	(177.49,191.11) --
	(177.97,191.36) --
	(178.44,191.62) --
	(178.92,191.88) --
	(179.39,192.13) --
	(179.87,192.38) --
	(180.34,192.63) --
	(180.82,192.88) --
	(181.29,193.13) --
	(181.77,193.38) --
	(182.24,193.62) --
	(182.72,193.87) --
	(183.19,194.10) --
	(183.67,194.35) --
	(184.15,194.58) --
	(184.62,194.82) --
	(185.10,195.06) --
	(185.57,195.29) --
	(186.05,195.53) --
	(186.52,195.76) --
	(187.00,195.99) --
	(187.47,196.22) --
	(187.95,196.45) --
	(188.42,196.68) --
	(188.90,196.90) --
	(189.37,197.13) --
	(189.85,197.35) --
	(190.32,197.57) --
	(190.80,197.79) --
	(191.27,198.01) --
	(191.75,198.22) --
	(192.22,198.44) --
	(192.70,198.66) --
	(193.17,198.87) --
	(193.65,199.08) --
	(194.12,199.29) --
	(194.60,199.50) --
	(195.07,199.71) --
	(195.55,199.92) --
	(196.02,200.13) --
	(196.50,200.33) --
	(196.97,200.54) --
	(197.45,200.74) --
	(197.92,200.94) --
	(198.40,201.15) --
	(198.88,201.35) --
	(199.35,201.54) --
	(199.83,201.74) --
	(200.30,201.94) --
	(200.78,202.13) --
	(201.25,202.32) --
	(201.73,202.52) --
	(202.20,202.71) --
	(202.68,202.90) --
	(203.15,203.08) --
	(203.63,203.27) --
	(204.10,203.46) --
	(204.58,203.64) --
	(205.05,203.83) --
	(205.53,204.01) --
	(206.00,204.19) --
	(206.48,204.37) --
	(206.95,204.55) --
	(207.43,204.73) --
	(207.90,204.91) --
	(208.38,205.09) --
	(208.85,205.26) --
	(209.33,205.44) --
	(209.80,205.61) --
	(210.28,205.78) --
	(210.75,205.96) --
	(211.23,206.13) --
	(211.70,206.30) --
	(212.18,206.47) --
	(212.65,206.64) --
	(213.13,206.81) --
	(213.60,206.97) --
	(214.08,207.14) --
	(214.56,207.30) --
	(215.03,207.47) --
	(215.51,207.63) --
	(215.98,207.79) --
	(216.46,207.95) --
	(216.93,208.11) --
	(217.41,208.27) --
	(217.88,208.43) --
	(218.36,208.59) --
	(218.83,208.74) --
	(219.31,208.90) --
	(219.78,209.06) --
	(220.26,209.21) --
	(220.73,209.36) --
	(221.21,209.51) --
	(221.68,209.66) --
	(222.16,209.81) --
	(222.63,209.97) --
	(223.11,210.12) --
	(223.58,210.26) --
	(224.06,210.41) --
	(224.53,210.56) --
	(225.01,210.70) --
	(225.48,210.85) --
	(225.96,210.99) --
	(226.43,211.13) --
	(226.91,211.27) --
	(227.38,211.41) --
	(227.86,211.55) --
	(228.33,211.69) --
	(228.81,211.83) --
	(229.29,211.97) --
	(229.76,212.11) --
	(230.24,212.24) --
	(230.71,212.38) --
	(231.19,212.52) --
	(231.66,212.65) --
	(232.14,212.78) --
	(232.61,212.91) --
	(233.09,213.05) --
	(233.56,213.18) --
	(234.04,213.31) --
	(234.51,213.44) --
	(234.99,213.57) --
	(235.46,213.70) --
	(235.94,213.82) --
	(236.41,213.95) --
	(236.89,214.07) --
	(237.36,214.20) --
	(237.84,214.32) --
	(238.31,214.44) --
	(238.79,214.57) --
	(239.26,214.69) --
	(239.74,214.80) --
	(240.21,214.92) --
	(240.69,215.04) --
	(241.16,215.16) --
	(241.64,215.28) --
	(242.11,215.40) --
	(242.59,215.52) --
	(243.06,215.64) --
	(243.54,215.76) --
	(244.02,215.87) --
	(244.49,215.98) --
	(244.97,216.10) --
	(245.44,216.21) --
	(245.92,216.32) --
	(246.39,216.44) --
	(246.87,216.55) --
	(247.34,216.66) --
	(247.82,216.77) --
	(248.29,216.88) --
	(248.77,216.99) --
	(249.24,217.10) --
	(249.72,217.20) --
	(250.19,217.31) --
	(250.67,217.42) --
	(251.14,217.53) --
	(251.62,217.63) --
	(252.09,217.74) --
	(252.57,217.84) --
	(253.04,217.95) --
	(253.52,218.05) --
	(253.99,218.15) --
	(254.47,218.26) --
	(254.94,218.36) --
	(255.42,218.46) --
	(255.89,218.56) --
	(256.37,218.66) --
	(256.84,218.76) --
	(257.32,218.86) --
	(257.79,218.96) --
	(258.27,219.05) --
	(258.74,219.15) --
	(259.22,219.25) --
	(259.70,219.34) --
	(260.17,219.44) --
	(260.65,219.53) --
	(261.12,219.63) --
	(261.60,219.72) --
	(262.07,219.82) --
	(262.55,219.91) --
	(263.02,220.00) --
	(263.50,220.10) --
	(263.97,220.19) --
	(264.45,220.28) --
	(264.92,220.37) --
	(265.40,220.46) --
	(265.87,220.55) --
	(266.35,220.64) --
	(266.82,220.73) --
	(267.30,220.81) --
	(267.77,220.90) --
	(268.25,220.99) --
	(268.72,221.07) --
	(269.20,221.16) --
	(269.67,221.25) --
	(270.15,221.33) --
	(270.62,221.41) --
	(271.10,221.50) --
	(271.57,221.58) --
	(272.05,221.66) --
	(272.52,221.75) --
	(273.00,221.83) --
	(273.47,221.91) --
	(273.95,221.99) --
	(274.43,222.07) --
	(274.90,222.15) --
	(275.38,222.23) --
	(275.85,222.31) --
	(276.33,222.39) --
	(276.80,222.47) --
	(277.28,222.55) --
	(277.75,222.63) --
	(278.23,222.70) --
	(278.70,222.77) --
	(279.18,222.85) --
	(279.65,222.93) --
	(280.13,223.00) --
	(280.60,223.08) --
	(281.08,223.16) --
	(281.55,223.23) --
	(282.03,223.30) --
	(282.50,223.38) --
	(282.98,223.45) --
	(283.45,223.53) --
	(283.93,223.60) --
	(284.40,223.67) --
	(284.88,223.75) --
	(285.35,223.82) --
	(285.83,223.89) --
	(286.30,223.96) --
	(286.78,224.03) --
	(287.25,224.10) --
	(287.73,224.17) --
	(288.20,224.24) --
	(288.68,224.31) --
	(289.16,224.38) --
	(289.63,224.44) --
	(290.11,224.51) --
	(290.58,224.58) --
	(291.06,224.65) --
	(291.53,224.71) --
	(292.01,224.78) --
	(292.48,224.85) --
	(292.96,224.91) --
	(293.43,224.98) --
	(293.91,225.04) --
	(294.38,225.11) --
	(294.86,225.17) --
	(295.33,225.23) --
	(295.81,225.30) --
	(296.28,225.36) --
	(296.76,225.42) --
	(297.23,225.49) --
	(297.71,225.55) --
	(298.18,225.61) --
	(298.66,225.67) --
	(299.13,225.73) --
	(299.61,225.79) --
	(300.08,225.85) --
	(300.56,225.91) --
	(301.03,225.97) --
	(301.51,226.03) --
	(301.98,226.09) --
	(302.46,226.15) --
	(302.93,226.21) --
	(303.41,226.27) --
	(303.89,226.32) --
	(304.36,226.38) --
	(304.84,226.44) --
	(305.31,226.49) --
	(305.79,226.55) --
	(306.26,226.61) --
	(306.74,226.66) --
	(307.21,226.72) --
	(307.69,226.77) --
	(308.16,226.83) --
	(308.64,226.88) --
	(309.11,226.93) --
	(309.59,226.99) --
	(310.06,227.04) --
	(310.54,227.10) --
	(311.01,227.15) --
	(311.49,227.20) --
	(311.96,227.25) --
	(312.44,227.30) --
	(312.91,227.36) --
	(313.39,227.41) --
	(313.86,227.46) --
	(314.34,227.51) --
	(314.81,227.56) --
	(315.29,227.62) --
	(315.76,227.66) --
	(316.24,227.71) --
	(316.71,227.76) --
	(317.19,227.81) --
	(317.66,227.86) --
	(318.14,227.91) --
	(318.61,227.96) --
	(319.09,228.01) --
	(319.57,228.06) --
	(320.04,228.11) --
	(320.52,228.16) --
	(320.99,228.20) --
	(321.47,228.25) --
	(321.94,228.30) --
	(322.42,228.34) --
	(322.89,228.39) --
	(323.37,228.44) --
	(323.84,228.48) --
	(324.32,228.53) --
	(324.79,228.57) --
	(325.27,228.62) --
	(325.74,228.66) --
	(326.22,228.71) --
	(326.69,228.75) --
	(327.17,228.80) --
	(327.64,228.84) --
	(328.12,228.88) --
	(328.59,228.93) --
	(329.07,228.97) --
	(329.54,229.01) --
	(330.02,229.06) --
	(330.49,229.10) --
	(330.97,229.14) --
	(331.44,229.19) --
	(331.92,229.23) --
	(332.39,229.27) --
	(332.87,229.31) --
	(333.34,229.35) --
	(333.82,229.39) --
	(334.30,229.43) --
	(334.77,229.47) --
	(335.25,229.51) --
	(335.72,229.55) --
	(336.20,229.59) --
	(336.67,229.63) --
	(337.15,229.67) --
	(337.62,229.71) --
	(338.10,229.75) --
	(338.57,229.79) --
	(339.05,229.83) --
	(339.52,229.87) --
	(340.00,229.90) --
	(340.47,229.94) --
	(340.95,229.98) --
	(341.42,230.02) --
	(341.90,230.06) --
	(342.37,230.10) --
	(342.85,230.13) --
	(343.32,230.17) --
	(343.80,230.21) --
	(344.27,230.24) --
	(344.75,230.28) --
	(345.22,230.32) --
	(345.70,230.35) --
	(346.17,230.39) --
	(346.65,230.43) --
	(347.12,230.46) --
	(347.60,230.50) --
	(348.07,230.54) --
	(348.55,230.57) --
	(349.03,230.61) --
	(349.50,230.64) --
	(349.98,230.68) --
	(350.45,230.71) --
	(350.93,230.75) --
	(351.40,230.78) --
	(351.88,230.81) --
	(352.35,230.85) --
	(352.83,230.88) --
	(353.30,230.91) --
	(353.78,230.95) --
	(354.25,230.98) --
	(354.73,231.01) --
	(355.20,231.04) --
	(355.68,231.07) --
	(356.15,231.11) --
	(356.63,231.14) --
	(357.10,231.17) --
	(357.58,231.20) --
	(358.05,231.23) --
	(358.53,231.26) --
	(359.00,231.30) --
	(359.48,231.33) --
	(359.95,231.36) --
	(360.43,231.39) --
	(360.90,231.42) --
	(361.38,231.45) --
	(361.85,231.48) --
	(362.33,231.51) --
	(362.80,231.54) --
	(363.28,231.57) --
	(363.76,231.60) --
	(364.23,231.63) --
	(364.71,231.66) --
	(365.18,231.69) --
	(365.66,231.72) --
	(366.13,231.75) --
	(366.61,231.78) --
	(367.08,231.81) --
	(367.56,231.84) --
	(368.03,231.86) --
	(368.51,231.89) --
	(368.98,231.92) --
	(369.46,231.95) --
	(369.93,231.98) --
	(370.41,232.00) --
	(370.88,232.03) --
	(371.36,232.06) --
	(371.83,232.09) --
	(372.31,232.11) --
	(372.78,232.14) --
	(373.26,232.17) --
	(373.73,232.19) --
	(374.21,232.22) --
	(374.68,232.25) --
	(375.16,232.27) --
	(375.63,232.30) --
	(376.11,232.33) --
	(376.58,232.35) --
	(377.06,232.38) --
	(377.53,232.40) --
	(378.01,232.43) --
	(378.48,232.45) --
	(378.96,232.48) --
	(379.44,232.50) --
	(379.91,232.53) --
	(380.39,232.55) --
	(380.86,232.57) --
	(381.34,232.60) --
	(381.81,232.62) --
	(382.29,232.65) --
	(382.76,232.67) --
	(383.24,232.70) --
	(383.71,232.72) --
	(384.19,232.74) --
	(384.66,232.77) --
	(385.14,232.79) --
	(385.61,232.81) --
	(386.09,232.84) --
	(386.56,232.86) --
	(387.04,232.89) --
	(387.51,232.91) --
	(387.99,232.93) --
	(388.46,232.95) --
	(388.94,232.98) --
	(389.41,233.00) --
	(389.89,233.02) --
	(390.36,233.04) --
	(390.84,233.07) --
	(391.31,233.09) --
	(391.79,233.11) --
	(392.26,233.13) --
	(392.74,233.16) --
	(393.21,233.18) --
	(393.69,233.20) --
	(394.17,233.22) --
	(394.64,233.24) --
	(395.12,233.26);
\end{scope}
\begin{scope}
\path[clip] (  0.00,  0.00) rectangle (433.62,289.08);
\definecolor{drawColor}{RGB}{0,0,0}

\path[draw=drawColor,line width= 0.4pt,line join=round,line cap=round] ( 49.20, 92.23) -- ( 49.20,220.79);

\path[draw=drawColor,line width= 0.4pt,line join=round,line cap=round] ( 49.20, 92.23) -- ( 43.20, 92.23);

\path[draw=drawColor,line width= 0.4pt,line join=round,line cap=round] ( 49.20,124.37) -- ( 43.20,124.37);

\path[draw=drawColor,line width= 0.4pt,line join=round,line cap=round] ( 49.20,156.51) -- ( 43.20,156.51);

\path[draw=drawColor,line width= 0.4pt,line join=round,line cap=round] ( 49.20,188.65) -- ( 43.20,188.65);

\path[draw=drawColor,line width= 0.4pt,line join=round,line cap=round] ( 49.20,220.79) -- ( 43.20,220.79);

\node[text=drawColor,rotate= 90.00,anchor=base,inner sep=0pt, outer sep=0pt, scale=  1.00] at ( 34.80, 92.23) {-0.08};

\node[text=drawColor,rotate= 90.00,anchor=base,inner sep=0pt, outer sep=0pt, scale=  1.00] at ( 34.80,124.37) {-0.06};

\node[text=drawColor,rotate= 90.00,anchor=base,inner sep=0pt, outer sep=0pt, scale=  1.00] at ( 34.80,156.51) {-0.04};

\node[text=drawColor,rotate= 90.00,anchor=base,inner sep=0pt, outer sep=0pt, scale=  1.00] at ( 34.80,188.65) {-0.02};

\node[text=drawColor,rotate= 90.00,anchor=base,inner sep=0pt, outer sep=0pt, scale=  1.00] at ( 34.80,220.79) {0.00};

\path[draw=drawColor,line width= 0.4pt,line join=round,line cap=round] ( 49.20, 61.20) --
	(408.42, 61.20) --
	(408.42,239.88) --
	( 49.20,239.88) --
	cycle;
\end{scope}
\begin{scope}
\path[clip] (  0.00,  0.00) rectangle (433.62,289.08);
\definecolor{drawColor}{RGB}{0,0,0}

\node[text=drawColor,anchor=base,inner sep=0pt, outer sep=0pt, scale=  1.20] at (228.81, 15.60) {q};
\end{scope}
\begin{scope}
\path[clip] (  0.00,  0.00) rectangle (433.62,289.08);
\definecolor{drawColor}{RGB}{0,0,0}

\node[text=drawColor,rotate= 90.00,anchor=base,inner sep=0pt, outer sep=0pt, scale=  1.20] at ( 10.80,150.54) {$J'_{+}(q)$};

\path[draw=drawColor,line width= 0.4pt,line join=round,line cap=round] ( 62.50, 61.20) -- (395.12, 61.20);

\path[draw=drawColor,line width= 0.4pt,line join=round,line cap=round] ( 62.50, 61.20) -- ( 62.50, 55.20);

\path[draw=drawColor,line width= 0.4pt,line join=round,line cap=round] (110.02, 61.20) -- (110.02, 55.20);

\path[draw=drawColor,line width= 0.4pt,line join=round,line cap=round] (157.54, 61.20) -- (157.54, 55.20);

\path[draw=drawColor,line width= 0.4pt,line join=round,line cap=round] (205.05, 61.20) -- (205.05, 55.20);

\path[draw=drawColor,line width= 0.4pt,line join=round,line cap=round] (300.08, 61.20) -- (300.08, 55.20);

\path[draw=drawColor,line width= 0.4pt,line join=round,line cap=round] (347.60, 61.20) -- (347.60, 55.20);

\path[draw=drawColor,line width= 0.4pt,line join=round,line cap=round] (395.12, 61.20) -- (395.12, 55.20);

\node[text=drawColor,anchor=base,inner sep=0pt, outer sep=0pt, scale=  1.00] at ( 62.50, 39.60) {0};

\node[text=drawColor,anchor=base,inner sep=0pt, outer sep=0pt, scale=  1.00] at (110.02, 39.60) {1};

\node[text=drawColor,anchor=base,inner sep=0pt, outer sep=0pt, scale=  1.00] at (157.54, 39.60) {2};

\node[text=drawColor,anchor=base,inner sep=0pt, outer sep=0pt, scale=  1.00] at (205.05, 39.60) {3};

\node[text=drawColor,anchor=base,inner sep=0pt, outer sep=0pt, scale=  1.00] at (300.08, 39.60) {5};

\node[text=drawColor,anchor=base,inner sep=0pt, outer sep=0pt, scale=  1.00] at (347.60, 39.60) {6};

\node[text=drawColor,anchor=base,inner sep=0pt, outer sep=0pt, scale=  1.00] at (395.12, 39.60) {7};
\end{scope}
\begin{scope}
\path[clip] ( 49.20, 61.20) rectangle (408.42,239.88);
\definecolor{drawColor}{RGB}{255,0,0}

\path[draw=drawColor,line width= 0.4pt,dash pattern=on 4pt off 4pt ,line join=round,line cap=round] (267.20, 61.20) -- (267.20,239.88);
\end{scope}
\begin{scope}
\path[clip] (  0.00,  0.00) rectangle (433.62,289.08);
\definecolor{drawColor}{RGB}{255,0,0}

\path[draw=drawColor,line width= 0.4pt,line join=round,line cap=round] (267.20, 61.20) -- (267.20, 61.20);

\path[draw=drawColor,line width= 0.8pt,line join=round,line cap=round] (267.20, 61.20) -- (267.20, 55.20);

\node[text=drawColor,anchor=base,inner sep=0pt, outer sep=0pt, scale=  1.00] at (267.20, 39.60) {$q^*$};
\end{scope}
\end{tikzpicture}

%% file: cvar.tex
\begin{tikzpicture}[x=1pt,y=1pt]
\definecolor{fillColor}{RGB}{255,255,255}
\path[use as bounding box,fill=fillColor,fill opacity=0.00] (0,0) rectangle (433.62,289.08);
\begin{scope}
\path[clip] ( 48.00, 60.00) rectangle (204.81,265.08);
\definecolor{drawColor}{RGB}{0,0,0}

\path[draw=drawColor,line width= 0.4pt,line join=round,line cap=round] ( 53.81,257.48) --
	( 58.65,257.48) --
	( 63.49,257.48) --
	( 68.33,257.48) --
	( 73.17,257.48) --
	( 78.01,257.48) --
	( 82.85,257.48) --
	( 87.69,257.48) --
	( 92.53,257.48) --
	( 97.37,257.48) --
	(102.21,257.48);

\path[draw=drawColor,line width= 0.4pt,line join=round,line cap=round] (102.21,177.99) --
	(107.05,172.84) --
	(111.89,168.48) --
	(116.73,164.67) --
	(121.57,161.38) --
	(126.41,158.49) --
	(131.24,155.94) --
	(136.08,153.67) --
	(140.92,151.65) --
	(145.76,149.86) --
	(150.60,148.25) --
	(155.44,146.81) --
	(160.28,145.51) --
	(165.12,144.33) --
	(169.96,143.26) --
	(174.80,143.26);

\path[draw=drawColor,line width= 0.4pt,line join=round,line cap=round] (174.80, 84.82) --
	(179.64, 81.08) --
	(184.48, 77.49) --
	(189.32, 74.06) --
	(194.16, 70.76) --
	(199.00, 67.60);
\end{scope}
\begin{scope}
\path[clip] (  0.00,  0.00) rectangle (433.62,289.08);
\definecolor{drawColor}{RGB}{0,0,0}

\path[draw=drawColor,line width= 0.4pt,line join=round,line cap=round] ( 48.00, 67.10) -- ( 48.00,255.79);

\path[draw=drawColor,line width= 0.4pt,line join=round,line cap=round] ( 48.00, 67.10) -- ( 42.00, 67.10);

\path[draw=drawColor,line width= 0.4pt,line join=round,line cap=round] ( 48.00, 94.06) -- ( 42.00, 94.06);

\path[draw=drawColor,line width= 0.4pt,line join=round,line cap=round] ( 48.00,121.01) -- ( 42.00,121.01);

\path[draw=drawColor,line width= 0.4pt,line join=round,line cap=round] ( 48.00,147.97) -- ( 42.00,147.97);

\path[draw=drawColor,line width= 0.4pt,line join=round,line cap=round] ( 48.00,174.92) -- ( 42.00,174.92);

\path[draw=drawColor,line width= 0.4pt,line join=round,line cap=round] ( 48.00,201.88) -- ( 42.00,201.88);

\path[draw=drawColor,line width= 0.4pt,line join=round,line cap=round] ( 48.00,228.83) -- ( 42.00,228.83);

\path[draw=drawColor,line width= 0.4pt,line join=round,line cap=round] ( 48.00,255.79) -- ( 42.00,255.79);

\node[text=drawColor,rotate= 90.00,anchor=base,inner sep=0pt, outer sep=0pt, scale=  1.00] at ( 37.20, 67.10) {3.0};

\node[text=drawColor,rotate= 90.00,anchor=base,inner sep=0pt, outer sep=0pt, scale=  1.00] at ( 37.20, 94.06) {3.5};

\node[text=drawColor,rotate= 90.00,anchor=base,inner sep=0pt, outer sep=0pt, scale=  1.00] at ( 37.20,121.01) {4.0};

\node[text=drawColor,rotate= 90.00,anchor=base,inner sep=0pt, outer sep=0pt, scale=  1.00] at ( 37.20,147.97) {4.5};

\node[text=drawColor,rotate= 90.00,anchor=base,inner sep=0pt, outer sep=0pt, scale=  1.00] at ( 37.20,174.92) {5.0};

\node[text=drawColor,rotate= 90.00,anchor=base,inner sep=0pt, outer sep=0pt, scale=  1.00] at ( 37.20,201.88) {5.5};

\node[text=drawColor,rotate= 90.00,anchor=base,inner sep=0pt, outer sep=0pt, scale=  1.00] at ( 37.20,228.83) {6.0};

\node[text=drawColor,rotate= 90.00,anchor=base,inner sep=0pt, outer sep=0pt, scale=  1.00] at ( 37.20,255.79) {6.5};

\path[draw=drawColor,line width= 0.4pt,line join=round,line cap=round] ( 48.00, 60.00) --
	(204.81, 60.00) --
	(204.81,265.08) --
	( 48.00,265.08) --
	cycle;
\end{scope}
\begin{scope}
\path[clip] (  0.00,  0.00) rectangle (216.81,289.08);
\definecolor{drawColor}{RGB}{0,0,0}

\node[text=drawColor,anchor=base,inner sep=0pt, outer sep=0pt, scale=  1.20] at (126.41, 14.40) {$\lambda$};

\node[text=drawColor,rotate= 90.00,anchor=base,inner sep=0pt, outer sep=0pt, scale=  1.20] at (  9.60,162.54) {$\mathrm{CVaR}_{\alpha}(Z^{*})$};
\end{scope}
\begin{scope}
\path[clip] (  0.00,  0.00) rectangle (433.62,289.08);
\definecolor{drawColor}{RGB}{0,0,0}

\path[draw=drawColor,line width= 0.4pt,line join=round,line cap=round] ( 53.81, 60.00) -- (204.81, 60.00);

\path[draw=drawColor,line width= 0.4pt,line join=round,line cap=round] ( 53.81, 60.00) -- ( 53.81, 54.00);

\path[draw=drawColor,line width= 0.4pt,line join=round,line cap=round] (150.60, 60.00) -- (150.60, 54.00);

\path[draw=drawColor,line width= 0.4pt,line join=round,line cap=round] (199.00, 60.00) -- (199.00, 54.00);

\node[text=drawColor,anchor=base,inner sep=0pt, outer sep=0pt, scale=  1.00] at ( 53.81, 42.00) {0.00};

\node[text=drawColor,anchor=base,inner sep=0pt, outer sep=0pt, scale=  1.00] at (150.60, 42.00) {0.02};

\node[text=drawColor,anchor=base,inner sep=0pt, outer sep=0pt, scale=  1.00] at (199.00, 42.00) {0.03};
\end{scope}
\begin{scope}
\path[clip] ( 48.00, 60.00) rectangle (204.81,265.08);
\definecolor{drawColor}{RGB}{255,0,0}

\path[draw=drawColor,line width= 0.4pt,dash pattern=on 4pt off 4pt ,line join=round,line cap=round] (105.00, 60.00) -- (105.00,265.08);

\path[draw=drawColor,line width= 0.4pt,dash pattern=on 4pt off 4pt ,line join=round,line cap=round] ( 48.00,174.92) -- (204.81,174.92);
\end{scope}
\begin{scope}
\path[clip] (  0.00,  0.00) rectangle (433.62,289.08);
\definecolor{drawColor}{RGB}{255,0,0}

\path[draw=drawColor,line width= 0.4pt,line join=round,line cap=round] (105.00, 60.00) -- (105.00, 60.00);

\path[draw=drawColor,line width= 0.8pt,line join=round,line cap=round] (105.00, 60.00) -- (105.00, 54.00);

\node[text=drawColor,anchor=base,inner sep=0pt, outer sep=0pt, scale=  1.00] at (105.00, 42.00) {$\lambda^*$};
\end{scope}
\begin{scope}
\path[clip] (264.81, 60.00) rectangle (421.62,265.08);
\definecolor{drawColor}{RGB}{0,0,0}

\path[draw=drawColor,line width= 0.4pt,line join=round,line cap=round] (270.62,257.48) --
	(275.62,257.48) --
	(280.63,257.48) --
	(285.64,257.48) --
	(290.64,257.48) --
	(295.65,257.48) --
	(300.66,257.48) --
	(305.66,257.48) --
	(310.67,257.48) --
	(315.68,257.48) --
	(320.68,239.15) --
	(325.69,223.15) --
	(330.70,208.81) --
	(335.70,196.06) --
	(340.71,184.51) --
	(345.72,174.09) --
	(350.73,164.56) --
	(355.73,155.83) --
	(360.74,147.92) --
	(365.75,140.65) --
	(370.75,133.95) --
	(375.76,127.74) --
	(380.77,121.98) --
	(385.77,116.65) --
	(390.78,111.63) --
	(395.79,102.07) --
	(400.79, 92.89) --
	(405.80, 84.13) --
	(410.81, 75.68) --
	(415.81, 67.60);
\end{scope}
\begin{scope}
\path[clip] (  0.00,  0.00) rectangle (433.62,289.08);
\definecolor{drawColor}{RGB}{0,0,0}

\path[draw=drawColor,line width= 0.4pt,line join=round,line cap=round] (265.61, 60.00) -- (415.81, 60.00);

\path[draw=drawColor,line width= 0.4pt,line join=round,line cap=round] (265.61, 60.00) -- (265.61, 54.00);

\path[draw=drawColor,line width= 0.4pt,line join=round,line cap=round] (290.64, 60.00) -- (290.64, 54.00);

\path[draw=drawColor,line width= 0.4pt,line join=round,line cap=round] (315.68, 60.00) -- (315.68, 54.00);

\path[draw=drawColor,line width= 0.4pt,line join=round,line cap=round] (340.71, 60.00) -- (340.71, 54.00);

\path[draw=drawColor,line width= 0.4pt,line join=round,line cap=round] (365.75, 60.00) -- (365.75, 54.00);

\path[draw=drawColor,line width= 0.4pt,line join=round,line cap=round] (390.78, 60.00) -- (390.78, 54.00);

\path[draw=drawColor,line width= 0.4pt,line join=round,line cap=round] (415.81, 60.00) -- (415.81, 54.00);

\node[text=drawColor,anchor=base,inner sep=0pt, outer sep=0pt, scale=  1.00] at (265.61, 42.00) {0.000};

\node[text=drawColor,anchor=base,inner sep=0pt, outer sep=0pt, scale=  1.00] at (315.68, 42.00) {0.010};

\node[text=drawColor,anchor=base,inner sep=0pt, outer sep=0pt, scale=  1.00] at (365.75, 42.00) {0.020};

\node[text=drawColor,anchor=base,inner sep=0pt, outer sep=0pt, scale=  1.00] at (415.81, 42.00) {0.030};

\path[draw=drawColor,line width= 0.4pt,line join=round,line cap=round] (264.81, 66.34) -- (264.81,259.32);

\path[draw=drawColor,line width= 0.4pt,line join=round,line cap=round] (264.81, 66.34) -- (258.81, 66.34);

\path[draw=drawColor,line width= 0.4pt,line join=round,line cap=round] (264.81, 93.91) -- (258.81, 93.91);

\path[draw=drawColor,line width= 0.4pt,line join=round,line cap=round] (264.81,121.48) -- (258.81,121.48);

\path[draw=drawColor,line width= 0.4pt,line join=round,line cap=round] (264.81,149.05) -- (258.81,149.05);

\path[draw=drawColor,line width= 0.4pt,line join=round,line cap=round] (264.81,176.62) -- (258.81,176.62);

\path[draw=drawColor,line width= 0.4pt,line join=round,line cap=round] (264.81,204.18) -- (258.81,204.18);

\path[draw=drawColor,line width= 0.4pt,line join=round,line cap=round] (264.81,231.75) -- (258.81,231.75);

\path[draw=drawColor,line width= 0.4pt,line join=round,line cap=round] (264.81,259.32) -- (258.81,259.32);

\node[text=drawColor,rotate= 90.00,anchor=base,inner sep=0pt, outer sep=0pt, scale=  1.00] at (254.01, 66.34) {3.0};

\node[text=drawColor,rotate= 90.00,anchor=base,inner sep=0pt, outer sep=0pt, scale=  1.00] at (254.01, 93.91) {3.2};

\node[text=drawColor,rotate= 90.00,anchor=base,inner sep=0pt, outer sep=0pt, scale=  1.00] at (254.01,121.48) {3.4};

\node[text=drawColor,rotate= 90.00,anchor=base,inner sep=0pt, outer sep=0pt, scale=  1.00] at (254.01,149.05) {3.6};

\node[text=drawColor,rotate= 90.00,anchor=base,inner sep=0pt, outer sep=0pt, scale=  1.00] at (254.01,176.62) {3.8};

\node[text=drawColor,rotate= 90.00,anchor=base,inner sep=0pt, outer sep=0pt, scale=  1.00] at (254.01,204.18) {4.0};

\node[text=drawColor,rotate= 90.00,anchor=base,inner sep=0pt, outer sep=0pt, scale=  1.00] at (254.01,231.75) {4.2};

\node[text=drawColor,rotate= 90.00,anchor=base,inner sep=0pt, outer sep=0pt, scale=  1.00] at (254.01,259.32) {4.4};

\path[draw=drawColor,line width= 0.4pt,line join=round,line cap=round] (264.81, 60.00) --
	(421.62, 60.00) --
	(421.62,265.08) --
	(264.81,265.08) --
	cycle;
\end{scope}
\begin{scope}
\path[clip] (216.81,  0.00) rectangle (433.62,289.08);
\definecolor{drawColor}{RGB}{0,0,0}

\node[text=drawColor,anchor=base,inner sep=0pt, outer sep=0pt, scale=  1.20] at (343.21, 14.40) {$\lambda$};

\node[text=drawColor,rotate= 90.00,anchor=base,inner sep=0pt, outer sep=0pt, scale=  1.20] at (226.41,162.54) {$q^*$};
\end{scope}
\end{tikzpicture}